\documentclass[12pt]{amsart}

\usepackage{amsmath,amsthm,amsfonts,amssymb,bm,graphicx }

\usepackage{graphicx}
\usepackage{enumitem}
\usepackage{xparse}
\usepackage{tikz}
\usepackage{comment}
\usepackage{bm}
\usepackage{amsaddr}
\usepackage{orcidlink}

\usepackage
{hyperref}
\hypersetup{colorlinks=true,citecolor=blue,linkcolor=blue,urlcolor=blue}

\makeatletter
\g@addto@macro\bfseries{\boldmath}
\makeatother

\theoremstyle{plain}
\newtheorem{theorem}{Theorem}
\newtheorem*{theorem*}{Theorem}
\newtheorem{lemma}[theorem]{Lemma}

\newtheorem{prop}[theorem]{Proposition}

\theoremstyle{remark}

\newtheorem{example}{Example}

\numberwithin{theorem}{section}

\numberwithin{equation}{section}

\def\N{\mathbb N}
\def\Z{\mathbb Z}
\def\R{\mathbb R}
\def\Q{\mathbb Q}
\def\C{\mathbb C}

\def\O{\mathcal O}

\def\bt{\blacktriangle}

\begin{document}

\author{Simona Fry\v{s}ov\'a$^1$ \orcidlink{0009-0000-8060-9307}, Magdal\'ena Tinkov\'a$^2$ \orcidlink{0000-0003-0874-9705}}

\title[Indecomposable quadratic forms over biquadratic and simplest cubic fields]{Additively indecomposable quadratic forms over biquadratic and simplest cubic fields}

\address{$^1$Charles University, Faculty of Mathematics and Physics, Department of Algebra,
Sokolovsk\'{a} 83, 18600 Praha 8, Czech Republic}
\email{simonkahlavinkova@gmail.com}

\address{$^2$Faculty of Information Technology, Czech Technical University in Prague, Th\'akurova 9, 160 00 Praha 6, Czech Republic}
\email{tinkova.magdalena@gmail.com}



\thanks{The authors were supported by Czech Science Foundation GA\v{C}R, grant 22-11563O. S. F. was further supported by Charles University programmes PRIMUS/24/SCI/010, GA UK 252931 and SVV-2025-260837 }

\begin{abstract} 
In this paper, we study additively indecomposable quadratic forms over real biquadratic and simplest cubic fields. In particular, we show that over these fields, we can always find such a classical form in 2 variables, which differs from the situation for forms over integers. Moreover, for some cases, we derive a lower bound on the number of classical, additively indecomposable binary quadratic forms up to equivalence.      
\end{abstract}

\setcounter{tocdepth}{2}  \maketitle 

\section{Introduction}

The study of universal quadratic forms definitely belongs among classical topics in number theory, and, moreover, we can highlight many revolutionary results coming up in the last few years. Over rationals, by a universal quadratic form, we mean a form with integer coefficients, which represents only positive integers (or zero if all entries are zero), and, furthermore, represents them all. A typical example is $x^2+y^2+z^2+w^2$, i.e., the sum of four squares, which was already considered by Lagrange.

In general, instead of $\Z$, we can consider the ring of algebraic integers $\O_K$ of any number field $K$. Then, the set of positive integers is replaced by the set of totally positive algebraic integers $\O_K^+$, i.e., of those elements of $\O_K$ with all conjugates positive. One of the most interesting directions of the study of such forms is their minimal number of variables -- which is, e.g., four for $\Z$. In this task, we can consider somewhat useful elements, so-called indecomposable integers, which are totally positive algebraic integers not representable as a sum of two elements of $\O_K^+$.

Besides algebraic integers, our quadratic form $Q$ can represent other quadratic forms. If this is true for all totally positive definite forms with $n$ variables, we say that $Q$ is $n$-universal. To study this property, we can use an analogy of indecomposable integers, i.e., additively indecomposable quadratic forms, which are totally positive definite quadratic forms that cannot be decomposed as a sum of two totally positive semi-definite quadratic forms. As we will see in this paper, indecomposable integers can also be used for the construction of such forms.

A more detailed study of additively indecomposable quadratic forms was initiated by Mordell \cite{Mo2}, who also provided one of the first key results. In particular, for a form $Q$ over $\Z$, he proved that if its determinant is large enough (more concretely, larger than the Hermite constant), then $Q$ is decomposable. A similar result for forms over number fields was first shown by Baeza and Icaza \cite{BI} and later refined by Tinkov\'a and Yatsyna \cite{TY}.

Another question we can ask about additively indecomposable quadratic forms is if, for a given $K$ and number of variables $n$, there exists such a form. And here, it is important to mention that for $2\leq n \leq 5$, there do not exist additively indecomposable (classical) quadratic forms over $\Z$ with this number of variables. Moreover, still for $K=\Q$, Erd\"{o}s and Ko \cite{EK1, EK2} studied in detail the existence of these forms for different choices of $n$ and values of the determinant of $Q$. In this spirit, in contrast to $\Z$, the authors of \cite{TY} showed that there always exists an additively indecomposable binary quadratic form over every real quadratic field.

The examination of the existence of these forms is also one of the main topics of this paper. We study it for two different families of totally real fields. The first of them is the family of biquadratic fields, for which we prove the following theorem:            

\begin{theorem} \label{thm:main1}
In every real biquadratic field, there exists a classical, additively indecomposable quadratic form in $2$ variables.
\end{theorem}

Therefore, this resembles the situation in quadratic fields. Moreover, in the proof, we indeed use forms that are additively indecomposable in quadratic subfields. 

The second family we focus on is the family of the so-called simplest cubic fields. They were first introduced by Shanks \cite{Sh} and studied intensively since then, e.g., in the relation to indecomposable integers \cite{KT}. We can define them as $K=\Q(\rho)$ where $\rho$ is the largest root of the polynomial $x^3-ax^2-(a+3)x-1$ with $a\in\Z_{\geq -1}$. For them, we prove the following result: 

\begin{theorem} \label{thm:main_simplest_existence}
If $\O_K=\Z[\rho]$, then there exist classical, additively indecomposable quadratic forms in $2$ and $3$ variables in $\Q(\rho)$. 
\end{theorem}

Note that $\O_K=\Z[\rho]$ occurs in infinitely many cases, e.g., whenever $a^2+3a+9$ is square-free.

Moreover, we can also ask about the number of such forms up to equivalence. In this case, we use some nice properties of indecomposable integers to provide a lower bound.

\begin{theorem} \label{thm:simplest_lower_bound}
Let $\Q(\rho)$ be a simplest cubic field with $\O_K=\Z[\rho]$ such that $a\geq 6$, and let $a=6A+a_0$ where $a_0\in\{0,1,2,3,4,5\}$. Up to equivalence, the number of classical, additively indecomposable binary quadratic forms in $\Q(\rho)$ is at least:
\begin{enumerate}
\item $\frac{27A^2+21A-6}{2}$ if $a_0=0$,
\item $\frac{27A^2+39A}{2}$ if $a_0=1$,
\item $\frac{27A^2+39A}{2}$ if $a_0=2$,
\item $\frac{27A^2+57A+12}{2}$ if $a_0=3$,
\item $\frac{27A^2+57A+18}{2}$ if $a_0=4$,
\item $\frac{27A^2+75A+36}{2}$ if $a_0=5$.
\end{enumerate}
\end{theorem}

The paper is organized as follows. In Section \ref{sec:preli}, we give some necessary notation and results used later in this paper. Section \ref{sec:nondec} focuses more on additively indecomposable quadratic forms; we introduce there some known results as well as a new result showing a possible type of additively indecomposable quadratic forms using indecomposable integers (Proposition \ref{prop:diam}). Section \ref{sec:biquadratic} is devoted to biquadratic fields and the proof of Theorem \ref{thm:main1}. For that, we first study general cases and, in Subsection \ref{subsec:special}, we resolve the remaining exceptional cases. In the end, in Section \ref{sec:simplest}, we turn our attention to the simplest cubic fields and prove both Theorems \ref{thm:main_simplest_existence} and \ref{thm:simplest_lower_bound}.

\section{Preliminaries} \label{sec:preli}

Let $K$ be a totally real number field of degree $d$, and let $\sigma_i$ for $1\leq i\leq d$ be its embeddings into $\R$. We say that an element $\alpha\in K$ is totally positive if $\sigma_i(\alpha)>0$ for all $1\leq i\leq d$. By the symbol $\text{Tr}_{K/\Q}(\alpha)$, we will denote the trace of $\alpha$ over $K$, i.e., $\text{Tr}_{K/\Q}(\alpha)=\sum_{i=1}^d\sigma_i(\alpha)$. Similarly, by the norm of $\alpha$ over $K$, we will mean $N_{K/\Q}(\alpha)=\prod_{i=1}^d\sigma_i(\alpha)$. 

By $\O_K$, we denote the ring of algebraic integers in $K$, and $\O_K^{+}$ stands for its subset of totally positive algebraic integers. We say that an element $\alpha\in \O_K^{+}$ is indecomposable if we cannot express it as a sum of two totally positive algebraic integers. We write $\alpha \succ \beta$ for $\alpha,\beta\in\O_K$ if $\alpha-\beta\in\O_K^{+}$, and we also use symbol $\succeq$ if we include equality. It is straightforward that if $\alpha\succeq \beta$, then $\text{Tr}_{K/\Q}(\alpha)\geq \text{Tr}_{K/\Q}(\beta)$ and $N_{K/\Q}(\alpha)\geq N_{K/\Q}(\beta)$.   

An $n$-ary quadratic form is an expression of the type \[\Q(x_1,\ldots,x_n)=\sum_{1\leq i\leq j\leq n}a_{ij}x_ix_j\]   
where $a_{ij}\in\O_K$. We say that $Q$ is classical if $2|a_{ij}$ for $i\neq j$; otherwise, we call it non-classical. A form $Q$ is totally positive semi-definite if $Q(x_1,\ldots,x_n)\in\O_K^{+}\cup\{0\}$ for all $x_i\in\O_K$, and it is, moreover, totally positive definite if $Q(x_1,\ldots,x_n)=0$ only if $x_1=x_2=\cdots=x_n=0$. Unless stated otherwise, in this paper, we always assume that our forms are classical. 

By the Gram matrix $M_Q$ of $Q$, we mean
\[
M_Q=\left(\begin{matrix}
a_{11} & \frac{a_{12}}{2} & \dots & \frac{a_{1n}}{2}\\
\frac{a_{12}}{2} & a_{22} & \dots & \frac{a_{2n}}{2}\\
\vdots & \vdots & \ddots & \vdots \\
\frac{a_{1n}}{2} & \frac{a_{2n}}{2} & \dots & a_{nn} 
\end{matrix}\right).
\]
The symbol $\det(Q)$ denotes the determinant of $Q$ and is equal to the determinant of $M_Q$. A totally positive definite quadratic form $Q$ is additively indecomposable if we cannot express it as $Q=Q_1+Q_2$ where $Q_1$ and $Q_2$ are totally positive semi-definite quadratic forms in $n$ variables such that $Q_1,Q_2\neq 0$. In this definition, we suppose that both $Q_1$ and $Q_2$ (as well as $Q$) are classical. However, we can also allow $Q_1$, $Q_2$ or $Q$ to be non-classical, which we do in Section \ref{sec:nondec}. However, in the rest of the paper, we consider only decomposition of (classical) forms into classical quadratic forms.   
We say that two quadratic forms $Q$ and $H$ with $n$ variables are equivalent if there is a matrix $N\in\O_K^{n\times n}$ such that $M_Q=NM_HN^{t}$ where $\det(N)$ is a unit in $\O_K$. From this, it is clear that if two quadratic forms $Q$ and $H$ are equivalent, then their determinants are associated, i.e., $\det(M_Q)=\varepsilon\det(M_H)$ for some unit $\varepsilon\in\O_K^+$.

\section{Indecomposability of forms} \label{sec:nondec}

In this part, we show how additively indecomposable quadratic forms in $K$ can look like. That was inspired by the following result obtained in \cite{TY}. 

\begin{prop}[{\cite[Proposition 5.10]{TY}}] \label{prop:inde_from_inde}
Let $K$ be a totally real number field, and let $Q(x,y)=\alpha x^2+\beta xy+\gamma y^2 \in \O_K[x,y]$ be totally positive definite. If $\alpha$ and $\gamma$ are indecomposable integers in $\O_K$ and $\beta\neq 0$, then $Q$ cannot be decomposed as a sum of non-classical, totally positive semi-definite quadratic forms.
\end{prop}

This proposition has the following direct implication.

\begin{prop}
Over every totally real field $K$, there exists a non-classical, totally positive definite quadratic form with $2$ variables which cannot be decomposed as a sum of non-classical, totally positive semi-definite quadratic forms.
\end{prop}

\begin{proof}
To prove that, it suffices to consider the quadratic form $x^2+xy+y^2$. 
The determinant of this form is $\frac{3}{4}$, so it is totally positive definite by Sylvester's criterion. Then, by Proposition \ref{prop:inde_from_inde}, it cannot be decomposed since $1$ is an indecomposable integer in every totally real number field. 
\end{proof}

For a larger number of variables, Proposition \ref{prop:inde_from_inde} can be generalized in the following way.

\begin{prop} \label{prop:diam}
Let $Q(x_1,\ldots,x_n)=\sum_{i=1}^n\alpha x_i^2+\sum_{j=1}^{n-1}\beta_{i, i+1}x_i x_{i+1}$ be a totally positive definite quadratic form over $\O_K$ where $\beta_{i,i+1}\neq 0$ and elements $\alpha_i$ are indecomposable integers in $\O_K$. Then $Q$ cannot be decomposed as a sum of non-classical, totally positive semi-definite quadratic forms. 
\end{prop}

\begin{proof}
In the following, we will use the fact that the only decomposition of elements $\alpha_i$ is as $\alpha_i+0$. To get a contradiction, let us suppose that there exist totally positive semi-definite quadratic forms $Q_1=\sum_{i=1}^n\alpha_i^{(1)}x_i^2+\sum_{1\leq i<j\leq n}\beta_{i,j}^{(1)}x_ix_j$ and $Q_2=\sum_{i=1}^n\alpha_i^{(2)}x_i^2+\sum_{1\leq i<j\leq n}\beta_{i,j}^{(2)}x_ix_j$ such that $Q=Q_1+Q_2$ and $Q_1,Q_2\neq 0$. Clearly, for all $1\leq i\leq n$, either $\alpha_i^{(1)}=\alpha_i$ and $\alpha_i^{(2)}=0$, or $\alpha_i^{(1)}=0$ and $\alpha_i^{(2)}=\alpha_i$.

Since both $Q_1$ and $Q_2$ are totally positive semi-definite and non-zero, without loss of generality, there exists an index $k$, $1\leq k\leq n$, such that $\alpha_k^{(1)}\neq 0$ and $\alpha_{k+1}^{(1)}= 0$. Moreover, $\beta_{k,k+1}\neq 0$, which implies that at least one of $\beta_{k,k+1}^{(1)}$ and $\beta_{k,k+1}^{(2)}$ is non-zero. Without loss of generality, suppose the first case. However, the principal submatrix
\[
\left(
\begin{matrix}
\alpha_k^{(1)} & \beta_{k,k+1}^{(1)}\\
\beta_{k,k+1}^{(1)} & 0
\end{matrix}
\right)
\]
of the Gram matrix of $Q_1$ is not totally positive semi-definite, which implies that $Q_1$ is not totally positive semi-definite, a contradiction.      
\end{proof}

Moreover, we can use Proposition \ref{prop:diam} for classical quadratic forms. If such a form satisfies the requirements of this proposition, it is not possible to decompose it into non-classical quadratic forms, which also covers classical forms, i.e., it is additively indecomposable in accordance with our definition.

\section{Real biquadratic fields} \label{sec:biquadratic}

In this part, we will consider real biquadratic fields, which are fields of the form $K=\Q(\sqrt{p},\sqrt{q})$ where $p,q>1$, $p\neq q$ are square-free. Moreover, the remaining square root lying in $K$ is $\sqrt{r}=\frac{\sqrt{pq}}{\gcd(p,q)}$. Thus, $K$ has quadratic subfields $\Q(\sqrt{p})$, $\Q(\sqrt{q})$ and $\Q(\sqrt{r})$. After possible intechanging of the role of $p, q$ and $r$, we distinguish the following cases of integral bases \cite{Ja,Wi}:
\begin{enumerate}
\item if $p\equiv 2\;(\text{mod }4)$ and $q\equiv 3\;(\text{mod }4)$, then $\O_K=\Z\Big[1,\sqrt{p},\sqrt{q},\frac{\sqrt{p}+\sqrt{r}}{2}\Big]$,
\item if $p\equiv 2\;(\text{mod }4)$ and $q\equiv 1\;(\text{mod }4)$, then $\O_K=\Z\Big[1,\sqrt{p},\frac{1+\sqrt{q}}{2},\frac{\sqrt{p}+\sqrt{r}}{2}\Big]$,
\item if $p\equiv 3\;(\text{mod }4)$ and $q\equiv 1\;(\text{mod }4)$, then $\O_K=\Z\Big[1,\sqrt{p},\frac{1+\sqrt{q}}{2},\frac{\sqrt{p}+\sqrt{r}}{2}\Big]$,
\item if $p,q\equiv 1\;(\text{mod }4)$, then
\begin{enumerate}
\item $\O_K=\Z\Big[1,\frac{1+\sqrt{p}}{2},\frac{1+\sqrt{q}}{2},\frac{1+\sqrt{p}+\sqrt{q}+\sqrt{r}}{4}\Big]$ if $\frac{p}{\gcd(p,q)}\equiv \frac{q}{\gcd(p,q)}\equiv 1\;(\text{mod }4)$,
\item $\O_K=\Z\Big[1,\frac{1+\sqrt{p}}{2},\frac{1+\sqrt{q}}{2},\frac{1-\sqrt{p}+\sqrt{q}+\sqrt{r}}{4}\Big]$ if $\frac{p}{\gcd(p,q)}\equiv \frac{q}{\gcd(p,q)}\equiv 3\;(\text{mod }4)$.
\end{enumerate}
\end{enumerate}
Let us consider $\alpha=x+y\sqrt{p}+z\sqrt{q}+w\sqrt{r}\in K$. Embeddings of $\Q(\sqrt{p},\sqrt{q})$ into $\C$ are given in the following way:
\begin{align*}
\sigma_1(\alpha)&=x+y\sqrt{p}+z\sqrt{q}+w\sqrt{r},\\
\sigma_2(\alpha)&=x-y\sqrt{p}+z\sqrt{q}-w\sqrt{r},\\
\sigma_3(\alpha)&=x+y\sqrt{p}-z\sqrt{q}-w\sqrt{r},\\
\sigma_4(\alpha)&=x-y\sqrt{p}-z\sqrt{q}+w\sqrt{r}.
\end{align*}
By \cite[Lemma 3.1]{CLSTZ}, if $\alpha$ is totally positive, then it holds that
\begin{equation} \label{eq:biqua_inequ}
x>|y|\sqrt{p},|z|\sqrt{q},|w|\sqrt{r}.
\end{equation}

It is not always true that indecomposable integers from $\Q(\sqrt{p})$, $\Q(\sqrt{q})$ and $\Q(\sqrt{r})$ remain indecomposable in $\Q(\sqrt{p},\sqrt{p})$. Some cases when they do are described in the following theorem:

\begin{theorem}[{\cite[Theorem 4.3]{Man}}] \label{thm:Man}
Let $p$ and $q$ be as above.
\begin{enumerate}
\item Suppose $K$ is a biquadratic field of Case (1--3), and $p < r$. Then the indecomposable integers from $\Q(\sqrt{p})$ and $\Q(\sqrt{q})$ remain indecomposable in $\Q(\sqrt{p},\sqrt{q})$.
\item Suppose $K$ is a biquadratic field of Case (4), and $p < q < r$, then the
indecomposable integers from $\Q(\sqrt{p})$ and $\Q(\sqrt{q})$ remain indecomposable in $\Q(\sqrt{p},\sqrt{q})$.
\end{enumerate}
\end{theorem}

In this section, we will prove Theorem \ref{thm:main1} from the introduction. For that, we use quadratic forms additively indecomposable in quadratic subfields. In particular, we will benefit from the following results.

\begin{prop}[{\cite[Proposition 5.2]{TY}}] \label{prop:exi3}
Let $D\equiv 3\;(\textup{mod }4)$. Then the quadratic form $Q(x,y)=2x^2+2\sqrt{D}xy+\frac{D+1}{2}y^2$ of determinant $\det(Q)=1$ is additively indecomposable over $\mathcal{O}_{\Q(\sqrt{D})}$.
\end{prop}

\begin{prop}[{\cite[Proposition 5.3]{TY}}] \label{prop:exi2}
Let $D\equiv 2\;(\textup{mod }4)$. Then the quadratic form \[Q(x,y)=2x^2+2(1+\sqrt{D})xy+\left(\frac{D}{2}+1+\sqrt{D}\right)y^2\] of determinant $\det(Q)=1$ is additively indecomposable over $\mathcal{O}_{\Q(\sqrt{D})}$.
\end{prop}

\begin{prop}[{\cite[Proposition 5.4]{TY}}] \label{prop:exi1}
Let $D>17$ be square-free. Then the following quadratic forms are additively indecomposable over $\mathcal{O}_{\Q(\sqrt{D})}$:
\begin{enumerate}
\item $Q(x,y)=3x^2+2(3+\sqrt{D})xy+\left(\frac{D+10}{3}+2\sqrt{D}\right)y^2$ with determinant $1$ for $D\equiv 5\;(\textup{mod }12)$,
\item $Q(x,y)=3x^2+2(3+\sqrt{D})xy+\left(\frac{D+11}{3}+2\sqrt{D}\right)y^2$ with determinant $2$ for $D\equiv 1\;(\textup{mod }12)$,
\item $Q(x,y)=4x^2+2(2+\sqrt{D})xy+\left(\frac{D+7}{4}+\sqrt{D}\right)y^2$ with determinant $3$ for $D\equiv 9\;(\textup{mod }12)$. 
\end{enumerate}
\end{prop}

In our proofs, we will need the following lemma, where by the decomposition of an element $\alpha\in\mathcal{O}_K$ we mean $\alpha=\beta +\gamma$ for some $\beta, \gamma\in\mathcal{O}_K\cup \{0\}$ :

\begin{lemma} \label{lem:decomp_biqua}
Let $K$ be a real biquadratic field. 
\begin{enumerate}
\item The number $2$ can be decomposed only as $1+1$ or $0+2$.
\item The number $3$ can be decomposed only as $1 + 2$, $0 + 3$ and $\frac{3+\sqrt{5}}{2}+\frac{3-\sqrt{5}}{2}$.
\item Up to embeddings in the corresponding number fields, the number $4$ can be decomposed only as $2+2$, $1+3$, $0+4$, $(2+\sqrt{2})+(2-\sqrt{2})$, $(2+\sqrt{3})+(2-\sqrt{3})$, $\frac{3+\sqrt{5}}{2}+\frac{5-\sqrt{5}}{2}$ and
\[
\left(2+\frac{\sqrt{2}+\sqrt{6}}{2}\right)+\left(2-\frac{\sqrt{2}+\sqrt{6}}{2}\right).
\]  
\end{enumerate}
\end{lemma}

\begin{proof}
Parts (1) and (2) are, in fact, stated in \cite[Lemma 4.3(2)--(3)]{KTZ}, (3) follows easily from \cite[Lemma 4.3(4)]{KTZ}.
\end{proof} 

\subsection{Counterexample} Before we start with the proof of the existence of an additively indecomposable binary quadratic form in biquadratic fields, we show a kind of the opposite situation to that in the following proofs. In particular, it is not always true that additively indecomposable quadratic forms from quadratic subfields do not decompose in the corresponding biquadratic field.

\begin{example}
Let us consider the biquadratic field $\Q(\sqrt{2},\sqrt{5})$, and let us take the element $7+2\sqrt{10}$, which is indecomposable in the quadratic subfield $\Q(\sqrt{10})$. From \cite[Example 4.6]{KTZ}, we know that in $\Q(\sqrt{2},\sqrt{5})$, the element $7+2\sqrt{10}$ decomposes as
\[
7+2\sqrt{10}=\frac{7}{2}+\sqrt{2}+\frac{1}{2}\sqrt{5}+\sqrt{10}+\frac{7}{2}-\sqrt{2}-\frac{1}{2}\sqrt{5}+\sqrt{10}.
\]
We use this example to construct a quadratic form additively indecomposable in $\Q(\sqrt{10})$ but decomposable in $\Q(\sqrt{2},\sqrt{5})$. Let us take $Q(x,y)=(7+2\sqrt{10})x^2+2xy+(7-2\sqrt{10})y^2$, which is additively indecomposable in $\Q(\sqrt{10})$ by Proposition \ref{prop:inde_from_inde}. However, in $\Q(\sqrt{2},\sqrt{5})$, we can write it as $Q=Q_1+Q_2$ where
\[
Q_1(x,y)=\left(\frac{7}{2}+\sqrt{2}+\frac{1}{2}\sqrt{5}+\sqrt{10}\right)x^2+2\left(\frac{1}{2}-\frac{1}{2}\sqrt{5}\right)xy+\left(\frac{7}{2}-\sqrt{2}+\frac{1}{2}\sqrt{5}-\sqrt{10}\right)y^2
\] 
and 
\[
Q_2(x,y)=\left(\frac{7}{2}-\sqrt{2}-\frac{1}{2}\sqrt{5}+\sqrt{10}\right)x^2+2\left(\frac{1}{2}+\frac{1}{2}\sqrt{5}\right)xy+\left(\frac{7}{2}+\sqrt{2}-\frac{1}{2}\sqrt{5}-\sqrt{10}\right)y^2.
\]
Easily, $\det(Q_1)=\det(Q_2)=0$, so $Q_1$ and $Q_2$ are both totally positive semi-definite.   
\end{example}    

\subsection{Cases (1)--(3)} \label{subsec:cases1_3}

First, we will focus on cases with bases (1)--(3) and prove that in such fields, we can always find an additively indecomposable binary quadratic form. In the following proofs, we sometimes omit technical details; some of them can be found in \cite{Hla}.

\begin{prop} \label{prop:ex_p_123}
Let $p$ and $q$ be as in Cases (1)--(3). Then,
\begin{enumerate}
\item if $p\equiv 2\;(\textup{mod }4)$, then the quadratic form $2x^2+2(1+\sqrt{p})xy+\left(\frac{p}{2}+1+\sqrt{p}\right)y^2$ is additively indecomposable in $\Q(\sqrt{p},\sqrt{q})$,
\item if $p\equiv 3\;(\textup{mod }4)$, then the quadratic form $2x^2+2\sqrt{p}xy+\frac{p+1}{2}y^2$ is additively indecomposable in $\Q(\sqrt{p},\sqrt{q})$. 
\end{enumerate} 
\end{prop}

\begin{proof}
We will show the proof only for the first part; the proof of the second part is similar and mostly uses techniques introduced in \cite[Proof of Proposition 5.2]{TY}, where the indecomposability of these forms in $\Q(\sqrt{p})$ is proved, and, moreover, can be found in \cite{Hla}.

From Lemma \ref{lem:decomp_biqua}, we know that $2$ decomposes only as $0+2$ or $1+1$ in $\Q(\sqrt{p},\sqrt{q})$. Thus, first, let us suppose that our quadratic form $Q$ can be expressed as $Q=Q_1+Q_2$ where
\begin{align*}
Q_1(x,y)&=x^2+2(1+\sqrt{p}-\beta)xy+\left(\frac{p}{2}+1+\sqrt{p}-\gamma\right)y^2,\\
Q_2(x,y)&=x^2+2\beta y^2+\gamma y^2
\end{align*} 
for some $\beta,\gamma\in\O_K$. Since $\det(Q)=1$, at least one of $\det(Q_1)$ and $\det(Q_2)$ must be zero, which follows from the inequality $\det(Q)\succeq \det(Q_1)+\det(Q_2)$. Let us say that $\det(Q_2)=0$. That gives $\gamma=\beta^2$. Then, we have
\[
\det(Q_1)=-\frac{p}{2}-\sqrt{p}-2\beta^2+2\beta+2\beta\sqrt{p}.
\] 
Put $\beta=\frac{b_1}{2}+\frac{b_2}{2}\sqrt{p}+\frac{b_3}{2}\sqrt{q}+\frac{b_4}{2}\sqrt{r}$ where $b_i\in\Z$ are such that $\beta\in\O_K$. Then
\[
\frac{1}{4}\text{Tr}_{K/\Q}(\det(Q_1))=-\frac{p}{2}+b_1-\frac{b_1^2}{2}+\left(b_2-\frac{b_2^2}{2}\right)p-\frac{b_3^2}{2}p-\frac{b_4^2}{2}r.
\]
Now, we will show that this expression is always negative. First, we see that $b_1-\frac{b_1^2}{2}>0$ only if $b_1=1$, and the same is true for $b_2-\frac{b_2^2}{2}$. Thus, if $b_1\neq 1$ and $b_2\neq 1$, then clearly $\text{Tr}_{K/\Q}(\det(Q_1))<0$, giving that $Q_1$ is not totally positive semi-definite. 

If $b_1=1$ and $b_2\neq 1$, then 
\[
-\frac{p}{2}+b_1-\frac{b_1^2}{2}+\left(b_2-\frac{b_2^2}{2}\right)p-\frac{b_3^2}{2}p-\frac{b_4^2}{2}r\leq -\frac{p}{2}+\frac{1}{2}<0
\] 
for $p\geq 2$.
If $b_1=1$ and $b_2=1$, then, considering forms of our integral bases, $b_4$ is odd and, thus, non-zero. We obtain
\[
-\frac{p}{2}+b_1-\frac{b_1^2}{2}+\left(b_2-\frac{b_2^2}{2}\right)p-\frac{b_3^2}{2}p-\frac{b_4^2}{2}r\leq -\frac{p}{2}+\frac{1}{2}+\frac{1}{2}p-\frac{1}{2}r=\frac{1}{2}-\frac{1}{2}r<0
\] 
for $r\geq 2$.

The second possible decomposition is of the type $Q=Q_3+Q_4$ where
\begin{align*}
Q_1(x,y)&=x^2+2(1+\sqrt{p})xy+\left(\frac{p}{2}+1+\sqrt{p}-\gamma\right)y^2,\\
Q_2(x,y)&=\gamma y^2
\end{align*} 
for some $\gamma\in\O_K^+$. In this case,
\[
\det(Q_1)=1-2\gamma.
\] 
Since $\frac{1}{4}\text{Tr}_{K/\Q}(\gamma)\geq 1$, we cannot have $\det(Q_1)\succeq 0$, which also excludes this case.  
\end{proof}

Note that in the previous proof, the roles of $p$ and $r$ are interchangeable. Therefore, we in fact get two additively indecomposable quadratic forms, one from $\Q(\sqrt{p})$ and one from $\Q(\sqrt{r})$. However, it would require further study to determine if these forms can be equivalent.

Using almost the same procedure, we can prove the following:

\begin{prop} \label{prop:ex_q_1}
Let $p\equiv 2\;(\textup{mod }4)$ and $q\equiv 3\;(\textup{mod }4)$. Then the quadratic form $2x^2+2\sqrt{q}xy+\frac{q+1}{2}y^2$ is additively indecomposable in $\Q(\sqrt{p},\sqrt{q})$.  
\end{prop}

\begin{proof}
    See \cite[Theorem 42]{Hla}.
\end{proof}

We see that in Proposition \ref{prop:ex_q_1}, we consider biquadratic fields with an integral base of Case (1). However, it is more difficult to prove a similar statement for Cases (2) and (3). Therefore, we show it under a restrictive condition.

\begin{prop} \label{prop:formq_23_1}
Let $p$ and $q$ be as in Cases (2) or (3) and, moreover, let $21<q<p,r$. Then the following quadratic forms are additively indecomposable in $\Q(\sqrt{p},\sqrt{q})$:
\begin{enumerate}
\item $Q(x,y)=3x^2+2(3+\sqrt{q})xy+\left(\frac{q+10}{3}+2\sqrt{q}\right)y^2$ with determinant $1$ for $q\equiv 5\;(\textup{mod }12)$,
\item $Q(x,y)=3x^2+2(3+\sqrt{q})xy+\left(\frac{q+11}{3}+2\sqrt{q}\right)y^2$ with determinant $2$ for $q\equiv 1\;(\textup{mod }12)$,
\item $Q(x,y)=4x^2+2(2+\sqrt{q})xy+\left(\frac{q+7}{4}+\sqrt{q}\right)y^2$ with determinant $3$ for $q\equiv 9\;(\textup{mod }12)$. 
\end{enumerate} 
\end{prop}

\begin{proof}
By Lemma \ref{lem:decomp_biqua}, we know that $3$ and $4$ can only decompose as a sum of two rational integers. To get a contradiction with the statement of the proposition, let us suppose that 
\[
Q(x,y)=Q_1(x,y)+Q_2(x,y)=(u_1x^2+2\beta_1 xy +\gamma_1y^2)+(u_2x^2+2\beta_2 xy +\gamma_2y^2)
\] 
where $u_i\in\N_0$, $\beta_i\in\O_K$ and $\gamma_i\in\O_K^+\cup\{0\}$ for $i=1,2$. Moreover, using similar arguments as at the end of the proof of Proposition \ref{prop:ex_p_123}, we can easily exclude $u_i=0$ for some $i$. Put $\beta_i=\frac{b_1^{(i)}}{2}+\frac{b_2^{(i)}}{2}\sqrt{p}+\frac{b_3^{(i)}}{2}\sqrt{q}+\frac{b_4^{(i)}}{2}\sqrt{r}$ and $\gamma_i=\frac{c_1^{(i)}}{2}+\frac{c_2^{(i)}}{2}\sqrt{p}+\frac{c_3^{(i)}}{2}\sqrt{q}+\frac{c_4^{(i)}}{2}\sqrt{r}$ where $b_j^{(i)},c_j^{(i)}\in\Z$ are such that $\beta,\gamma\in\O_K$. Then
\begin{multline} \label{eq:det_23_1}
\frac{1}{4}\text{Tr}_{K/\Q}(\det(Q_1)+\det(Q_2))=\frac{u_1c_1^{(1)}}{2}+\frac{u_2c_1^{(2)}}{2}-\frac{(b_1^{(1)})^2}{4}-\frac{(b_1^{(2)})^2}{4}-\frac{(b_2^{(1)})^2p}{4}-\frac{(b_2^{(2)})^2p}{4}\\-\frac{(b_3^{(1)})^2q}{4}-\frac{(b_3^{(2)})^2q}{4}-\frac{(b_4^{(1)})^2r}{4}-\frac{(b_4^{(2)})^2r}{4}\geq 0.
\end{multline} 
From the form of $Q$, it is clear that $b_2^{(1)}=-b_2^{(2)}$ and $b_4^{(1)}=-b_4^{(2)}$. Moreover, for all the forms in the statement, easily, we have
\begin{equation} \label{eq:estim_trac_23_1}
\frac{u_1c_1^{(1)}}{2}+\frac{u_2c_1^{(2)}}{2}\leq q+11.
\end{equation}
If $|b_2^{(i)}|\geq 2$ for some $i$, (\ref{eq:det_23_1}) and (\ref{eq:estim_trac_23_1}) implies
\[
0\leq q+11-\frac{(b_2^{(i)})^2p}{2}\leq q+11-2q=11-q<0
\]
for $p>q>21$, which is a contradiction. Thus, $|b_2^{(i)}|\leq 1$. Using an analogous argument, the same is also true for $b_4^{(i)}$. 

If $|b_3^{(i)}|\geq 3$ for some $i$, (\ref{eq:det_23_1}) and (\ref{eq:estim_trac_23_1}) gives
\[
0\leq q+11-\frac{(b_3^{(i)})^2q}{4}\leq q+11-\frac{9q}{4}=11-\frac{5q}{4}<0,
\]
which is impossible. Thus, $|b_3^{(i)}|\leq 2$ for $i=1,2$. Moreover, from the type of our forms, we see that $\frac{b_3^{(1)}}{2}+\frac{b_3^{(2)}}{2}=1$, which implies that $(b_3^{(1)},b_3^{(2)})\in\{(0,2),(1,1),(2,0)\}$. In particular, it follows that
\[
\frac{(b_3^{(1)})^2q}{4}+\frac{(b_3^{(2)})^2q}{4}\geq \frac{q}{2}.
\]

If $|b_2^{(i)}|=1$, then $|b_4^{(i)}|=1$ and vice versa. Similarly, as before, in this case, we obtain
\[
0\leq q+11-\frac{p}{2}-\frac{q}{2}-\frac{r}{2}<q+11-\frac{3q}{2}=11-\frac{q}{2}<0
\] 
for $q>21$. That means that $b_2^{(i)}=b_4^{(i)}=0$ for $i=1,2$, and, thus, $\beta_i\in\Q(\sqrt{q})$. Then
\[
\det(Q_i)=\frac{u_ic_1^{(i)}}{2}-\frac{(b_1^{(i)})^2}{4}-\frac{(b_3^{(i)})^2q}{4}+\frac{u_ic_2^{(i)}}{2}\sqrt{p}+\frac{u_ic_3^{(i)}-b_1^{(i)}b_3^{(i)}}{2}\sqrt{q}+\frac{u_ic_4^{(i)}}{2}\sqrt{r},
\] 
which needs to be totally positive or zero. Moreover, since $\det(Q)\in\{1,2,3\}$, Lemma~\ref{lem:decomp_biqua} implies $\det(Q_i)\in\Z$. Since $u_i\neq 0$, we have $c_2^{(i)}=c_4^{(i)}=0$ for $i=1,2$. That implies $\gamma_i\in\Q(\sqrt{q})$, and with $\beta_i\in\Q(\sqrt{q})$ from above, it follows that $Q_1$ and $Q_2$ are quadratic forms over $\Q(\sqrt{q})$. However, by \cite[Proposition 5.4]{TY}, the forms in the statement are additively indecomposable in $\Q(\sqrt{q})$, which finishes the proof.
\end{proof}

Thus, as we can see, we sometimes get additively indecomposable binary quadratic forms from all three subfields $\Q(\sqrt{p})$, $\Q(\sqrt{q})$, and $\Q(\sqrt{r})$ -- which can still be equivalent.

\subsection{Case (4)} \label{subsec:case4}

Now, we will focus on the biquadratic fields for which $p,q\equiv 1\;(\text{mod }4)$. In this case, the situation does not seem to be so simple as in the previous subsection. We first prove the existence under the conditions $21<p,q,r$, and the remaining cases are treated in Subsection \ref{subsec:special}. Note that in this case, $p$, $q$, and $r$ are interchangeable, so we can assume $p<q<r$.

\begin{prop} \label{prop:111}
Let $p,q\equiv 1\;(\textup{mod }4)$ and $21<p<q<r$. Then the following quadratic forms are additively indecomposable in $\Q(\sqrt{p},\sqrt{q})$:
\begin{enumerate}
\item $Q(x,y)=3x^2+2(3+\sqrt{p})xy+\left(\frac{p+10}{3}+2\sqrt{p}\right)y^2$ with determinant $1$ for $p\equiv 5\;(\textup{mod }12)$,
\item $Q(x,y)=3x^2+2(3+\sqrt{p})xy+\left(\frac{p+11}{3}+2\sqrt{p}\right)y^2$ with determinant $2$ for $p\equiv 1\;(\textup{mod }12)$,
\item $Q(x,y)=4x^2+2(2+\sqrt{p})xy+\left(\frac{p+7}{4}+\sqrt{p}\right)y^2$ with determinant $3$ for $p\equiv 9\;(\textup{mod }12)$. 
\end{enumerate} 
\end{prop} 

\begin{proof}
Similarly to the proof of Proposition \ref{prop:formq_23_1}, the numbers $3$ and $4$ have only decompositions lying in $\Z$. Therefore, every possible decomposition of $Q$ needs to be of the form
\[
Q(x,y)=Q_1(x,y)+Q_2(x,y)=(u_1x^2+2\beta_1 xy +\gamma_1y^2)+(u_2x^2+2\beta_2 xy +\gamma_2y^2)
\] 
where $u_i\in\N_0$, $\beta_i\in\O_K$ and $\gamma_i\in\O_K^+\cup\{0\}$ for $i=1,2$. Moreover, similarly as before, we can exclude cases with $u_i=0$ for some $i$. Now, let us write $\beta_i=\frac{b_1^{(i)}}{4}+\frac{b_2^{(i)}}{4}\sqrt{p}+\frac{b_3^{(i)}}{4}\sqrt{q}+\frac{b_4^{(i)}}{4}\sqrt{r}$ and $\gamma_i=\frac{c_1^{(i)}}{4}+\frac{c_2^{(i)}}{4}\sqrt{p}+\frac{c_3^{(i)}}{4}\sqrt{q}+\frac{c_4^{(i)}}{4}\sqrt{r}$ where $b_j^{(i)},c_j^{(i)}\in\Z$ are such that $\beta_i,\gamma_i\in\O_K$. Then
\begin{multline} \label{eq:det_111}
\frac{1}{4}\text{Tr}_{K/\Q}(\det(Q_1)+\det(Q_2))=\frac{u_1c_1^{(1)}}{4}+\frac{u_2c_1^{(2)}}{4}-\frac{(b_1^{(1)})^2}{16}-\frac{(b_1^{(2)})^2}{16}-\frac{(b_2^{(1)})^2p}{16}-\frac{(b_2^{(2)})^2p}{16}\\-\frac{(b_3^{(1)})^2q}{16}-\frac{(b_3^{(2)})^2q}{16}-\frac{(b_4^{(1)})^2r}{16}-\frac{(b_4^{(2)})^2r}{16}\geq 0.
\end{multline}
Moreover, again $b_3^{(1)}=-b_3^{(2)}$, $b_4^{(1)}=-b_4^{(2)}$ and 
\begin{equation} \label{eq:estim_trace_111}
\frac{u_1c_1^{(1)}}{4}+\frac{u_2c_1^{(2)}}{4}\leq p+11.
\end{equation}

We see that if $|b_2^{(i)}|\geq 5$ for some $i$, then (\ref{eq:det_111}) and (\ref{eq:estim_trace_111}) imply
\[
0\leq p+11-\frac{25p}{16}=11-\frac{9p}{16}p<0
\]
for $p>21$. That gives $|b_2^{(i)}|\leq 4$ for $i=1,2$. The condition $\frac{b_2^{(1)}}{16}+\frac{b_2^{(2)}}{16}=1$ coming from the type of our forms implies $(b_2^{(1)},b_2^{(2)})\in\{(0,4),(1,3),(2,2),(3,1),(4,0)\}$ where the cases $(0,4)$ and $(4,0)$, and $(1,3)$ and $(3,1)$ are interchangeable, so we can restrict to $\{(0,4),(1,3),(2,2)\}$.

First, let us consider $(b_2^{(1)},b_2^{(2)})=(0,4)$. If $b_3^{(1)}$ is non-zero, then the form of our integral basis implies that it is even. However, then (\ref{eq:det_111}) and (\ref{eq:estim_trace_111}) give
\[
0\leq p+11-p-\frac{4q}{16}-\frac{4q}{16}<11-\frac{p}{2}<0
\]  
whenever $p>21$, $p\equiv1\;(\text{mod }4)$. We can use a similar argument for $b_4^{(1)}$. Together, we get $b_3^{(1)}=b_3^{(2)}=b_4^{(1)}=b_4^{(2)}=0$ in this case. Then
\[
\det(Q_1)=\frac{u_1c_1^{(1)}}{4}-\frac{(b_1^{(1)})^2}{16}+\frac{u_1c_2^{(1)}}{4}\sqrt{p}+\frac{u_1c_3^{(1)}}{4}\sqrt{q}+\frac{u_1c_4^{(1)}}{4}\sqrt{r}
\] 
and
\[
\det(Q_2)=\frac{u_2c_1^{(2)}}{4}-p-\frac{(b_1^{(2)})^2}{16}+\frac{u_2c_2^{(2)}-2b_1^{(2)}}{4}\sqrt{p}+\frac{u_2c_3^{(2)}}{4}\sqrt{q}+\frac{u_2c_4^{(2)}}{4}\sqrt{r}.
\]
Moreover, since $\det(Q)\in\{1,2,3\}$, Lemma \ref{lem:decomp_biqua} implies $\det(Q_1),\det(Q_2)\in\Z$.
We also have $u_1,u_2\neq 0$, so $c_3^{(1)}=c_3^{(2)}=c_4^{(1)}=c_4^{(2)}=0$. Thus, $Q_1$ and $Q_2$ are forms over $\Q(\sqrt{p})$, and since $Q$ is additively indecomposable over this subfield \cite[Proposition 5.4]{TY}, the case $(b_2^{(1)},b_2^{(2)})=(0,4)$ does not give a suitable decomposition.

Now, suppose that $(b_2^{(1)},b_2^{(2)})=(1,3)$. Then $b_3^{(i)}$ and $b_4^{(i)}$ are odd, and it is not so hard to prove (using (\ref{eq:det_111}) and (\ref{eq:estim_trace_111})) that $|b_3^{(i)}|=|b_4^{(i)}|=1$. Now, we will consider traces of the determinants of $Q_1$ and $Q_2$ for different choices of $u_1$ and $u_2$. 
\begin{enumerate}
\item Let $Q$ be as in (1) or (2). In this case, $u_1+u_2=3$.
\begin{enumerate}
\item If $u_1=1$, then, to obtain $\frac{1}{4}\text{Tr}_{K/\Q}(\det(Q_1))\geq 0$, we need to have 
\[
\frac{u_1c_1^{(1)}}{4}-\frac{(b_1^{(1)})^2}{16}-\frac{p}{16}-\frac{q}{16}-\frac{r}{16}\geq 0,
\]
which implies $\frac{c_1^{(1)}}{4}>\frac{3p}{16}$, and, in a similar manner, using $u_2=2$, we get $\frac{c_1^{(2)}}{4}>\frac{11p}{32}$. Together, we get 
\[
\frac{p+11}{3}\geq \frac{c_1^{(1)}}{4}+\frac{c_1^{(2)}}{4}>\frac{17p}{32},
\]
which is impossible for $p>21$.
\item If $u_1=2$, then $u_2=1$. To get $\frac{1}{4}\text{Tr}_{K/\Q}(\det(Q_1)),\frac{1}{4}\text{Tr}_{K/\Q}(\det(Q_2))\geq 0$, one need to have $\frac{u_1c_1^{(1)}}{4}=\frac{2c_1^{(1)}}{4}>\frac{3p}{16}$ and $\frac{u_2c_1^{(2)}}{4}=\frac{c_1^{(2)}}{4}>\frac{11p}{16}$. Then
\[
\frac{p+11}{3}=\frac{c_1^{(1)}}{4}+\frac{c_1^{(2)}}{4}>\frac{3p}{32}+\frac{11p}{16}=\frac{25p}{32},
\]
which is not satisfied for $p>21$.   
\end{enumerate}
\item Now, let $Q$ be as in (3), so $u_1+u_2=4$. 
\begin{enumerate}
\item If $u_1=1$, then, similarly to before, we can see that
\[
\frac{p+7}{4}=\frac{c_1^{(1)}}{4}+\frac{c_1^{(2)}}{4}>\frac{3p}{16}+\frac{11p}{48}=\frac{5p}{12}
\]
cannot be true for $p>21$. 
\item If $u_1=u_2=2$, using $\frac{1}{4}\text{Tr}_{K/\Q}(\det(Q_2))\geq 0$, we have $2\frac{p+7}{4}>\frac{11p}{16}$, which is not satisfied for $p>21$.
\item In the end, let $u_1=3$ and $u_2=1$. Then, similarly to the previous part, we need $\frac{p+7}{4}>\frac{11p}{16}$, which is not possible. 
\end{enumerate}
\end{enumerate}
By this, we excluded $(b_2^{(1)},b_2^{(2)})=(1,3)$.

Finally, let $(b_2^{(1)},b_2^{(2)})=(2,2)$. In this case, $b_3^{(i)}$ and $b_4^{(i)}$ are even for $i=1,2$. Using the same techniques as before, it is easy to prove that $|b_3^{(i)}|,|b_4^{(i)}|\in\{0,2\}$, and, moreover, at least one of $b_3^{(i)}$ and $b_4^{(i)}$ is zero. If both of them are zero, we can use a similar procedure as for $(b_2^{(1)},b_2^{(2)})=(0,4)$ to conclude that $Q_1$ and $Q_2$ are forms over $\Q(\sqrt{p})$, which is impossible. So, let us now suppose that exactly one of $b_3^{(i)}$ and $b_4^{(i)}$ is zero, i.e., equal to $2$ in the absolute value. Now, we again distinguish different cases of $Q$.
\begin{enumerate}
\item Let $Q$ be as in (1) or (2). Then, without loss of generality, we can suppose $u_1=1$. That gives
\begin{multline*}
0\leq \frac{1}{4}\text{Tr}_{K/\Q}(\det(Q_1))=\frac{u_1c_1^{(1)}}{4}-\frac{(b_1^{(1)})^2}{16}-\frac{p}{4}-\frac{(b_3^{(1)})^2q}{16}-\frac{(b_3^{(1)})^2r}{16}\\<\frac{p+11}{3}-\frac{p}{4}-\frac{p}{4}=\frac{11}{3}-\frac{p}{6}<0
\end{multline*}
for $p>21$.  
\item Now, let $Q$ be as in (4). Then $u_1+u_2=4$, and it suffices to study $(u_1,u_2)\in\{(1,3),(2,2)\}$.
\begin{enumerate}
\item If $u_1=1$, a similar argument implies $\frac{p+7}{4}>\frac{p}{2}$, which is not true for $p>21$.
\item If $u_1=u_2=2$, we obtain
\[
\frac{p+7}{4}=\frac{c_1^{(1)}}{4}+\frac{c_1^{(2)}}{4}>\frac{p}{4}+\frac{p}{4}=\frac{p}{2},
\]
which leads to the same conclusion.
\end{enumerate}
\end{enumerate}
By this, we excluded the last case, and the proof is completed.          
\end{proof}

\subsection{Exceptional cases} \label{subsec:special}

In Subsection \ref{subsec:cases1_3}, we completely resolved the case when $p\equiv 2,3\;(\text{mod }4)$. Most of the cases with $p\equiv q\equiv 1\;(\text{mod }4)$ were studied in Subsection~\ref{subsec:case4}. Thus, it remains to look at those biquadratic fields, for which $p\in\{5,13,17,21\}$ and $q\equiv 1\;(\text{mod }4)$.

\subsubsection{$\sqrt{5}\in K$}

The first exceptional cases that we will discuss are biquadratic fields containing $\sqrt{5}$.

\begin{prop} \label{prop:sqrt5}
Let $q\equiv 1\;(\textup{mod }4)$, $q>5$ square-free. Then the quadratic form $2x^2+2xy+(3+\sqrt{5})y^2$ is additively indecomposable in $\Q(\sqrt{5},\sqrt{q})$. 
\end{prop}

\begin{proof}
Note that $2x^2+2xy+(3+\sqrt{5})y^2$ is additively indecomposable in $\Q(\sqrt{5})$ \cite[Theorem~5.30]{TY}. Moreover, since $5$ is a prime, we can assume $\gcd(5,q)=1$; otherwise, we can just switch $q$ and $r$. 

First, from Lemma \ref{lem:decomp_biqua}, we know that $2$ can be decomposed only as $0+2$ and $1+1$, so the summands always lie in $\Q(\sqrt{5})$. Now, we will discuss the decompositions of $3+\sqrt{5}$. Let us suppose that
\[
3+\sqrt{5}=\frac{a_1}{4}+\frac{b_1}{4}\sqrt{5}+\frac{c_1}{4}\sqrt{q}+\frac{d_1}{4}\sqrt{5q}+\frac{a_2}{4}+\frac{b_2}{4}\sqrt{5}+\frac{c_2}{4}\sqrt{q}+\frac{d_2}{4}\sqrt{5q}
\]
for some $a_i,b_i,c_i,d_i\in \Z$ such that $\frac{a_i}{4}+\frac{b_i}{4}\sqrt{5}+\frac{c_i}{4}\sqrt{q}+\frac{d_i}{4}\sqrt{5q}\in \O_K$ for $i=1,2$. Then, by (\ref{eq:biqua_inequ}), if $\frac{a_i}{4}+\frac{b_i}{4}\sqrt{5}+\frac{c_i}{4}\sqrt{q}+\frac{d_i}{4}\sqrt{5q}$ is totally positive, then
\[
a_i>|b_i|\sqrt{5},\;|c_i|\sqrt{q},\;|d_i|\sqrt{5q}.
\]
If $d_i\neq 0$ for some $i$, then both $d_1$ and $d_2$ are non-zero, which gives
\[
12=a_1+a_2>|d_1|\sqrt{5q}+|d_2|\sqrt{5q}\geq 2\sqrt{5q}\geq 2\sqrt{65}>16,
\]
a contradiction. Therefore, $d_1=d_2=0$, and by the form of the integral basis, we have $a_i=2\overline{a_i}$, $b_i=2\overline{b_i}$, and $c_i=2\overline{c_i}$ for some $\overline{a_i},\overline{b_i},\overline{c_i}\in\Z$. Then, if $\overline{c_1}\neq 0$, in which case $\overline{c_2}\neq 0$, as before, we obtain
\[
6=\overline{a_1}+\overline{a_2}>|\overline{c_1}|\sqrt{q}+|\overline{c_2}|\sqrt{q}\geq 2\sqrt{13}>7,
\]
which is not possible. Thus, $\overline{c_1}=\overline{c_2}=0$, and all decompositions of $3+\sqrt{5}$ lie in $\Q(\sqrt{5})$.  

Now, to get a contradiction, let us suppose that
\[
2x^2+2xy+(3+\sqrt{5})y^2=Q_1(x,y)+Q_2(x,y)=(u_1x^2+2\beta_1 xy+\gamma_1 y^2)+(u_2x^2+2\beta_2 xy+\gamma_2 y^2)
\]
for some $u_1,u_2\in\{0,1,2\}$, $\beta_i=\frac{b_1^{(i)}}{4}+\frac{b_2^{(i)}}{4}\sqrt{5}+\frac{b_3^{(i)}}{4}\sqrt{q}+\frac{b_4^{(i)}}{4}\sqrt{5q}$ for $b_j^{(i)}\in\Z$ satisfying the corresponding congruence conditions, and $\gamma_i=\frac{c_1^{(i)}}{2}+\frac{c_2^{(i)}}{2}\sqrt{5}$ where $c_j^{(i)}\in\Z$ and $c_1^{(i)}\equiv c_2^{(i)}\;(\text{mod }2)$ for $i=1,2$. As before, we need to have $u_i\alpha_i-\beta_i^2\succeq 0$, and we will study the trace of these elements. In particular, to get $\frac{1}{4}\text{Tr}_{K/\Q}(\det(Q_1)+\frac{1}{4}\text{Tr}_{K/\Q}(\det(Q_2)\geq 0$, we need to have
\begin{multline} \label{eq:ineq5}
6\geq \frac{u_1c_1^{(1)}}{2}+\frac{u_2c_1^{(2)}}{2}\geq \frac{(b_1^{(1)})^2}{16}+\frac{5(b_2^{(1)})^2}{16}+\frac{q(b_3^{(1)})^2}{16}+\frac{5q(b_4^{(1)})^2}{16}\\+\frac{(b_1^{(2)})^2}{16}+\frac{5(b_2^{(2)})^2}{16}+\frac{q(b_3^{(2)})^2}{16}+\frac{5q(b_4^{(2)})^2}{16}
\end{multline}
If $b_4^{(1)}\neq 0$, and, thus, at the same time $b_4^{(2)}\neq 0$, the inequality in (\ref{eq:ineq5}) implies 
\[
6\geq \frac{5q}{8}\geq \frac{65}{8}>8, 
\] 
a contradiction. Thus, $b_4^{(1)}=b_4^{(2)}=0$, and $b_3^{(i)}=2\overline{b}_3^{(i)}$ for some $\overline{b}_3^{(i)}\in\Z$. Similarly, if $\overline{b}_3^{(i)}\neq 0$, (\ref{eq:ineq5}) gives
\[
6\geq \frac{q}{2}\geq \frac{13}{2}>6.
\] 
It follows that $\overline{b}_3^{(i)}=0$ for $i=1,2$. Thus, the only possible decompositions of $2x^2+2xy+(3+\sqrt{5})y^2$ are over $\Q(\sqrt{5})$, and since this form is additively indecomposable over this subfield, it completes the proof.  
\end{proof} 

\subsubsection{$\sqrt{13}\in K$}

In a similar manner, we will discuss biquadratic fields containing $\sqrt{13}$. Note that in fact, the $\Q(\sqrt{5},\sqrt{13})$ is covered by Proposition \ref{prop:sqrt5}, so we can assume that $q>13$. Again, we will consider a form additively indecomposable over $\Q(\sqrt{13})$. 

\begin{prop} \label{prop:sqrt13}
Let $q\equiv 1\;(\textup{mod }4)$, $q>13$ square-free. Then the quadratic form $2x^2+2\left(\frac{1}{2}+\frac{1}{2}\sqrt{13}\right)xy+3y^2$ is additively indecomposable in $\Q(\sqrt{13},\sqrt{q})$.
\end{prop}

\begin{proof}
First, we can again suppose $\gcd(13,q)=1$.
By Lemma \ref{lem:decomp_biqua}, we see that $2$ and $3$ has only decompositions lying in $\Z$ (and thus in $\Q(\sqrt{13})$). To get a contradiction, let us suppose that
\[
2x^2+2\left(\frac{1}{2}+\frac{1}{2}\sqrt{13}\right)xy+3y^2=(u_1x^2+2\beta_1xy+v_1y^2)+(u_2x^2+2\beta_2xy+v_2y^2)
\]
where $u_i\in\{0,1,2\}$ and $v_i\in\{0,1,2,3\}$. Using the same expression for $\beta_i$ as in the proof of Proposition \ref{prop:sqrt5} and considering the sum of traces of determinants of the above forms, we get
\begin{multline} \label{eq:ineq13}
6\geq u_1v_1+u_2v_2\geq \frac{(b_1^{(1)})^2}{16}+\frac{13(b_2^{(1)})^2}{16}+\frac{q(b_3^{(1)})^2}{16}+\frac{13q(b_4^{(1)})^2}{16}\\+\frac{(b_1^{(2)})^2}{16}+\frac{13(b_2^{(2)})^2}{16}+\frac{q(b_3^{(2)})^2}{16}+\frac{13q(b_4^{(2)})^2}{16}. 
\end{multline}
Similarly as before, if $b_4^{(1)},b_4^{(2)}\neq 0$, we get
\[
6\geq \frac{13q}{8}\geq \frac{221}{8}>27,
\] 
which is impossible. Thus $b_4^{(1)}=b_4^{(2)}=0$, and both of coefficients $b_3^{(i)}$ are even. If they are also nonzero, (\ref{eq:ineq13}) implies
\[
6\geq \frac{q}{2}\geq \frac{17}{2}>8.
\]
Thus, $b_3^{(1)}=b_3^{(2)}=0$, and every decomposition of $2x^2+2\left(\frac{1}{2}+\frac{1}{2}\sqrt{13}\right)xy+3y^2$ is over $\Q(\sqrt{13})$, which is a contradiction with \cite[Proof of Theorem 5.5]{TY}.    
\end{proof}

\subsubsection{$\sqrt{17}\in K$}

Now, we will focus on the case $\sqrt{17}\in K$. Since the fields containing $\sqrt{5}$ and $\sqrt{13}$ were studied before, we can assume $q>17$.

\begin{prop} \label{prop:sqrt17}
Let $q\equiv 1\;(\textup{mod }4)$, $q>17$ square-free. Then the quadratic form $\left(\frac{5}{2}+\frac{1}{2}\sqrt{17}\right)x^2+2xy+\left(\frac{5}{2}-\frac{1}{2}\sqrt{17}\right)y^2$ is additively indecomposable in $\Q(\sqrt{17},\sqrt{q})$.
\end{prop}

\begin{proof}
First, note that $\frac{5}{2}+\frac{1}{2}\sqrt{17}$ (and thus also $\frac{5}{2}-\frac{1}{2}\sqrt{17}$) is indecomposable in $\Q(\sqrt{17})$, and therefore the form in the statement is additively indecomposable in $\Q(\sqrt{17})$ by Proposition~\ref{prop:inde_from_inde}. Since $17<q,r$, by Theorem \ref{thm:Man},  the element $\frac{5}{2}+\frac{1}{2}\sqrt{17}$ is also indecomposable in $\Q(\sqrt{17},\sqrt{q})$. Then the result follows again from Proposition \ref{prop:inde_from_inde}. 
\end{proof}

\subsubsection{$\sqrt{21}\in K$}

The last case we need to resolve is when $\sqrt{21}\in K$.

\begin{prop} \label{prop:sqrt21}
Let $q\equiv 1\;(\textup{mod }4)$, $q>21$ square-free, and let $r>21$. Then the quadratic form $2x^2+(3+\sqrt{21})xy+(5+\sqrt{21})y^2$ is additively indecomposable in $\Q(\sqrt{21},\sqrt{q})$.
\end{prop}

\begin{proof}
Compared to the previous cases, not necessarily $\gcd(21,q)=1$, but we can suppose $\gcd(21,q)\in\{1,3\}$. Note that the quadratic form $2x^2+\left(3+\sqrt{21}\right)xy+(5+\sqrt{21})y^2$ is additively indecomposable in $\Q(\sqrt{21})$ by \cite[Theorem 5.35]{TY}.

As before, $2$ has only decompositions with integer summands. Suppose that
\[
2x^2+(3+\sqrt{21})xy+(5+\sqrt{21})y^2=(u_1x^2+2\beta_1 xy+\gamma_1 y^2)+(u_2x^2+2\beta_2 xy+\gamma_2 y^2)
\]
for some $u_i\in\{0,1,2\}$, $\beta_i\in\O_K$ and $\gamma_i\in\O_K^{+}\cup\{0\}$. Now, we will study elements $\gamma_i=\frac{c_1^{(i)}}{4}+\frac{c_2^{(i)}}{4}\sqrt{p}+\frac{c_3^{(i)}}{4}\sqrt{q}+\frac{c_4^{(i)}}{4}\sqrt{r}$ for $c_j^{(i)}\in\Z$ such that $\gamma_i\in\O_K$. Let us suppose that $c_4^{(1)}=-c_4^{(2)}\neq 0$. First, assume $\gcd(p,q)=1$. Then (\ref{eq:biqua_inequ}) gives
\[
20\geq |c_4^{(i)}|\sqrt{r}\geq \sqrt{21\cdot 29}>24.
\]
Now, let $\gcd(21,q)=3$, i.e., we can suppose $q=3s$ and $r=7s$ where $s\equiv 3\;(\text{mod }4)$ is square-free positive integer such that $\gcd(21,s)=1$ and, moreover, suppose first that $q>11$. Then (\ref{eq:biqua_inequ}) implies
\[
20= c_1^{(1)}+c_1^{(2)}>(|c_4^{(1)}|+|c_4^{(2)}|)\sqrt{r}\geq 2\sqrt{7\cdot 19}>23,
\]
a contradiction.
Therefore, if $s\neq 11$, we have $c_4^{(1)}=c_4^{(2)}=0$. For the field $\Q(\sqrt{21},\sqrt{33})$, we can consider the quadratic form
\[
\left(5+2\frac{1+\sqrt{33}}{2}\right)x^2+2xy+\left(5+2\frac{1-\sqrt{33}}{2}\right)y^2,
\]
which is totally positive definite.
The element $5+2\frac{1+\sqrt{33}}{2}$ is indecomposable in $\Q(\sqrt{33})$, and, thus, by Theorem \ref{thm:Man}, is also indecomposable in $\Q(\sqrt{21},\sqrt{33})$. Then, the indecomposability of the above form follows from Proposition \ref{prop:inde_from_inde}.

So, in the rest of the proof, we can assume that $K\neq \Q(\sqrt{21},\sqrt{33})$. If $c_4^{(1)}=c_4^{(2)}=0$, then the coefficients $c_3^{(i)}$ are even. If $c_3^{(1)}=-c_3^{(2)}\neq 0$ and $\gcd(21,q)=1$, then
\[
20= c_1^{(1)}+c_1^{(2)}>(|c_3^{(1)}|+|c_3^{(2)}|)\sqrt{q}\geq 4\sqrt{29}>21.
\]
If $\gcd(21,q)=3$, we similarly get
\[
20\geq c_1^{(1)}+c_1^{(2)}>(|c_3^{(1)}|+|c_3^{(2)}|)\sqrt{q}\geq 4\sqrt{3\cdot 19}>30.
\]
Therefore, $c_3^{(1)}=c_3^{(2)}=0$, and $\gamma_i\in\Q(\sqrt{21})$. Let us denote $\gamma_i=\frac{\overline{c}_1^{(i)}}{2}+\frac{\overline{c}_2^{(i)}}{2}\sqrt{21}$ where coefficients $\overline{c}_j^{(i)}\in\Z$ are such that $\gamma_i\in\O_K$ for $i=1,2$.

Similarly as in the proof of Proposition \ref{prop:sqrt5}, we have
\begin{multline} \label{eq:ineq21}
10\geq \frac{u_1\overline{c}_1^{(1)}}{2}+\frac{u_2\overline{c}_1^{(2)}}{2}\geq \frac{(b_1^{(1)})^2}{16}+\frac{21(b_2^{(1)})^2}{16}+\frac{q(b_3^{(1)})^2}{16}+\frac{r(b_4^{(1)})^2}{16}\\+\frac{(b_1^{(2)})^2}{16}+\frac{21(b_2^{(2)})^2}{16}+\frac{q(b_3^{(2)})^2}{16}+\frac{r(b_4^{(2)})^2}{16},
\end{multline} 
where we use the same expression for $\beta_i$ as before. Considering the type of our quadratic form, we also have $b_3^{(1)}=-b_3^{(2)}$ and $b_4^{(1)}=-b_4^{(2)}$. From the previous part, we know that $r\geq \min\{609,133\}=133$ and $q\geq \min\{29,57\}=29$. If $b_4^{(1)}\neq 0$. Then (\ref{eq:ineq21}) implies
\[
10\geq \frac{r(b_4^{(1)})^2}{16}+\frac{r(b_4^{(2)})^2}{16}\geq\frac{133}{8}>16.
\]
Thus, as before, $b_4^{(1)}=b_4^{(2)}=0$, giving that $b_3^{(1)}$ and $b_3^{(2)}$ are even. If $b_3^{(1)}\neq 0$, (\ref{eq:ineq21}) gives
\[
10\geq \frac{q(b_3^{(1)})^2}{16}+\frac{q(b_3^{(2)})^2}{16}\geq\frac{29}{2}>14,
\]
from which $b_3^{(1)}=b_3^{(2)}=0$ follows. Thus, $\beta_i\in\Q(\sqrt{21})$, and since $2x^2+(3+\sqrt{21})xy+(5+\sqrt{21})y^2$ is additively indecomposable over $\Q(\sqrt{21})$, the proof is completed.  
\end{proof}

\subsection{Proof of Theorem \ref{thm:main1}}

Now, we are able to prove Theorem \ref{thm:main1} from the introduction.

\begin{theorem}[= Theorem \ref{thm:main1}] 
In every real biquadratic field, there exists a classical, additively indecomposable quadratic form in $2$ variables.
\end{theorem}

\begin{proof}
To show this result, we use quadratic forms additively indecomposable in quadratic subfields. In particular, regarding the integral bases, Cases (1)--(3) are discussed in Proposition \ref{prop:ex_p_123}, and Case (4) in Proposition \ref{prop:111}, with some exceptions. The remaining cases are resolved in Proposition~\ref{prop:sqrt5} ($\sqrt{5}\in K$), Proposition~\ref{prop:sqrt13} ($\sqrt{13}\in K$), Proposition~\ref{prop:sqrt17} ($\sqrt{17}\in K$) and Proposition~\ref{prop:sqrt21} ($\sqrt{21}\in K$). 
\end{proof}

\section{The simplest cubic fields} \label{sec:simplest}

In this section, we will study additively indecomposable quadratic forms for one family of cubic fields, the so-called simplest cubic fields. They are of the form $\Q(\rho)$ where $\rho$ is the largest root of the polynomial $x^3-ax^2-(a+3)x-1$ with $a\in\Z$, $a\geq -1$ \cite{Sh}. They possess many useful properties; more concretely, they are Galois, and all their totally positive units are squares. Moreover, for the conjugates of $\rho$, it holds that $a+1<\rho$, $-2<\rho'<-1$ and $-1<\rho''<0$. In \cite{KT}, the authors found all indecomposable integers in $\Z[\rho]$:

\begin{theorem}[{\cite[Theorem 1.2]{KT}}]
Let $K$ be the simplest cubic field with
$a\in\Z_{\geq -1}$ such that $\O_K=\Z[\rho]$. 
The elements $1$, $1+\rho+\rho^2$, and $-v-w\rho+(v+1)\rho^2$ where $0\leq v\leq a$ and $v(a+2)+1\leq w\leq (v+1)(a+1)$ are, up to multiplication by totally positive units, all the indecomposable integers in $\Q(\rho)$.
\end{theorem}

If we denote the set of elements $\alpha(v,w)=-v-w\rho+(v+1)\rho^2$ with $0\leq v\leq a$ and $v(a+2)+1\leq w\leq (v+1)(a+1)$ by $\bt$, we can define two transformations $T_1,T_2:\bt\rightarrow\bt$ in the following way:
\begin{multline*}
T_1(\alpha(v,v(a+2)+W+1))=(\alpha(v,v(a+2)+W+1))'(\rho')^{-2}\\=\alpha(W,W(a+2)+a-v-W+1)
\end{multline*}
and
\begin{multline*}
T_2(\alpha(v,v(a+2)+W+1))=(\alpha(v,v(a+2)+W+1))''\rho^{2}\\=\alpha(a-v-W,(a-v-W)(a+2)+v+1)
\end{multline*}
where $W=w-v(a+2)-1$, which implies $0\leq W\leq a-v$.

In most of this section, we will consider binary quadratic forms of a particular type. For them, we need to consider decompositions of numbers $2$ and $3$.

\begin{lemma} \label{lem:dec23_simplest}
Let $K$ be a simplest cubic field. Then $2$ can only be decomposed as $0+2$ and $1+1$ in $\Z[\rho]^+\cup\{0\}$. Similarly, $3$ can only be decomposed as $0+3$ and $1+2$ in $\Z[\rho]^+\cup\{0\}$. 
\end{lemma}

\begin{proof}
Since the first part follows from the second, we only prove the result for $3$. For $-1\leq a\leq 6$, we can use a computer program to verify the statement. The code we used is available at https://sites.google.com/view/tinkovamagdalena/codes. Therefore, we can assume $a\geq 7$. 

Let us suppose that $3=\alpha_1+\alpha_2$ where $\alpha_1,\alpha_2\in\Z[\rho]^+$. Then, by \cite[Lemma 2.1]{KY1}, we have
\[
3=\sqrt[3]{N_{K/\Q}(3)}\geq \sqrt[3]{N_{K/\Q}(\alpha_1)}+\sqrt[3]{N_{K/\Q}(\alpha_2)}.
\]
That implies $1\leq N_{K/\Q}(\alpha_i)\leq 8$. However, we know that if $\alpha_i$ is not associated with a rational integer, then $N_{K/\Q}(\alpha_i)\geq 2a+3>8$ \cite[Theorem 1]{LP}. Thus, $\alpha_1$ and $\alpha_2$ are associated with rational integers, giving $N_{K/\Q}(\alpha_i)\in\{1,8\}$. 

Easily, we can write $\alpha_i=n_i\varepsilon_i$, where $n_i\in\{1,2\}$ and $\varepsilon_i\in\Z[\rho]^+$ is a unit. However, for $a\geq 7$, if $\varepsilon_i\neq 1$, then there exists an embedding $\sigma$ of $K$ into $\C$ such that $\sigma(\varepsilon_i)>a^2\geq 49$ \cite[Lemma 6.2]{KT}. That gives
\[
\sigma(3)=3>n_i\sigma(\alpha_i)>49,
\]
which is impossible. Thus, we have $\alpha_1,\alpha_2\in\Z$, which completes the proof.   
\end{proof}

That allows us to construct some additively indecomposable binary quadratic forms in $K$.

\begin{prop} \label{prop:form_of_form23}
Let $\rho$ be as above.
\begin{enumerate}
\item If $\alpha$ is indecomposable integer in $\Z[\rho]$ and $2\alpha\succ 1$, then $2x^2+2xy+\alpha y^2$ is additively indecomposable over $\Z[\rho]$. 
\item If $\alpha$ is indecomposable integer in $\Z[\rho]$, $3\alpha\succ 1$ and $2\alpha\not\succeq 1$, then $3x^2+2xy+\alpha y^2$ is additively indecomposable over $\Z[\rho]$.
\end{enumerate}
\end{prop} 

\begin{proof}
We will only prove (2) since (1) can be shown in the same way. Let us take $\alpha$ as specified in (2). Then, clearly, $3x^2+2xy+\alpha y^2$ is totally positive definite. By Lemma \ref{lem:dec23_simplest}, the only possible decompositions of the form $3x^2+2xy+\alpha y^2$ can be of the form
\[
3x^2+2xy+\alpha y^2=ux^2+2xy+\alpha y^2+(3-u)x^2
\] 
where $u\in\{1,2\}$. However, $ux^2+2xy+\alpha y^2$ is never totally positive semi-definite for these cases of $u$. That completes the proof.    
\end{proof} 

Note that in the above proof, we only use the fact that we know the decompositions of $2$ and $3$ in the simplest cubic fields; otherwise, it is independent of the choice of the field. Thus, it passes whenever, in a field, the only decompositions of $2$ and $3$ are of this form.   

In the simplest cubic fields, for example, as $\alpha$, we can take $\alpha=1+\rho+\rho^2$ or $\alpha\in\bt$. In the latter case, using the transformations $T_1$ and $T_2$ defined above, we can create more additively indecomposable quadratic forms over $\Z[\rho]$. Concretely, if $ux^2+2xy+\alpha y^2$ is additively indecomposable for $u\in\{2,3\}$ and $\alpha\in\bt$, then $u(\rho')^{-2}x^2+2(\rho')^{-2}xy+(\rho')^{-2}\alpha'y^2$ and $u\rho^{2}x^2+2\rho^{2}xy+\rho^{2}\alpha''y^2$ are also additively indecomposable.

\subsection{Existence}

In this subsection, we will prove the existence of at least one additively indecomposable quadratic form in $2$ and $3$ variables over $\Z[\rho]$.

\subsubsection{Binary quadratic forms}

We will start with binary forms, for which we use a type considered in Proposition \ref{prop:form_of_form23}(1) and, as a suitable indecomposable integer, we will take $\alpha=1+\rho+\rho^2$.

\begin{prop} \label{prop:simplest_ex2}
The quadratic form $2x^2+2xy+(1+\rho+\rho^2)y^2$ is additively indecomposable over $\Z[\rho]$.
\end{prop}

\begin{proof}
By Proposition \ref{prop:form_of_form23}(1), the form $2x^2+2xy+(1+\rho+\rho^2)y^2$ is additively indecomposable if $2(1+\rho+\rho^2)\succ 1$. To prove that, let us consider the polynomial $g(x)=2x^2+2x+2$. The function $g$ attains its minimum for $x=-\frac{1}{2}$, for which $g\left(-\frac{1}{2}\right)=\frac{3}{2}>1$. That means that $g(x)>1$ for all $x\in\R$, including $x=\rho,\rho',\rho''$, which implies $2(1+\rho+\rho^2)\succ 1$. 
\end{proof}

\subsubsection{Ternary quadratic forms}

Now, we will show that the element $1+\rho+\rho^2$ can also be used to show that there exists an additively indecomposable quadratic form in $3$ variables in $\Q(\rho)$.

\begin{prop} \label{prop:simplest_ex3}
The quadratic form $Q(x,y,z)=(1+\rho+\rho^2)x^2+(1+\rho'+\rho'^2)y^2+(1+\rho''+\rho''^2)z^2+2xy+2yz$ is additively indecomposable over $\Z[\rho]$.
\end{prop}

\begin{proof}
If $Q$ is totally positive definite, it is additively indecomposable by Proposition \ref{prop:diam}. By Sylvester's criterion, we need to check if all leading principal minors are totally positive. The first one is $1+\rho+\rho^2$, which is totally positive. The second one is $(1+\rho+\rho^2)(1+\rho'+\rho'^2)-1$, which is a root of 
\[
x^3-(2a^2+6a+15)x^2+(a^4+6a^3+23a^2+42a+48)x-(2a^2+6a+17),
\]
i.e., this principal minor is totally positive, e.g., by Descartes' sign rule. Finally, for the determinant of $Q$, we get
\begin{align*}
\det(Q)&=(1+\rho+\rho^2)(1+\rho'+\rho'^2)(1+\rho''+\rho''^2)-(1+\rho+\rho^2)-(1+\rho''+\rho''^2)\\
&=N_ {K/\Q}(1+\rho+\rho^2)-\text{Tr}_{K/\Q}(1+\rho+\rho^2)+1+\rho'+\rho'^2\\
&=a^2+3a+9-(a^2+3a+9)+1+\rho'+\rho'^2\\
&=1+\rho'+\rho'^2\succ 0.
\end{align*}
Therefore, $Q$ is totally positive definite, which completes the proof.
\end{proof}

\begin{proof}[Proof of Theorem \ref{thm:main_simplest_existence}]
The proof of Theorem \ref{thm:main_simplest_existence} follows from Propositions \ref{prop:simplest_ex2} and \ref{prop:simplest_ex3}. 
\end{proof}

\subsection{Lower bound}

In this section, we will use forms from Proposition \ref{prop:form_of_form23} to find a lower bound on the number of non-equivalent classical, additively indecomposable quadratic forms over $\Z[\rho]$. For that, we need to find suitable elements $\alpha\in\bt$ for which conditions in the first or second part of Proposition \ref{prop:form_of_form23} are satisfied. 

\begin{prop} \label{prop:2alpha}
Let $\alpha(v,w)\in\bt$. If $0\leq v\leq \frac{a-1}{2}$ and $\frac{(2v+1)(a+2)+1}{2}\leq w\leq (v+1)(a+1)$, then $2\alpha(v,w)\succ 1$. 
\end{prop}

\begin{proof}
We see that
\[
2\alpha(v,w)-1=-(2v+1)-2w\rho+(2v+2)\rho^2.
\]
Thus, it remains to show that the element on the right side is totally positive and, in particular, that is true if it lies in $\bt$. We have
\[
1\leq 2v+1\leq a
\]
for $0\leq v\leq \frac{a-1}{2}$. Moreover, from $\frac{(2v+1)(a+2)+1}{2}\leq w\leq (v+1)(a+1)$, we obtain
\[
(2v+1)(a+2)+1\leq 2w\leq (2v+2)(a+1).
\]
Therefore, $-(2v+1)-2w\rho+(2v+2)\rho^2\in\bt$, from which the statement follows.      
\end{proof} 

Let us note that the remaining elements $\alpha(v,w)\in\bt$ do not satisfy $2\alpha(v,w)\succ 1$, and a detailed discussion can be found in \cite{Hla}.

Now, we will proceed with $3\alpha(v,w)$.

\begin{prop} \label{prop:3alpha}
Let $K=\Q(\rho)$ be a simplest cubic field with $a\geq 6$.
\begin{enumerate}
\item If $\overline{v}\in\Z$ is such that $\frac{a-3}{2}<\overline{v}\leq a-2$, $\frac{a+\overline{v}}{3}\in\Z$, then there exists $\overline{w}$, $\overline{v}(a+2)+1\leq \overline{w}\leq (a+1)(\overline{v}+1)$ such that $\frac{a^2+3a+3+\overline{w}}{3}\in\Z$, 
\[
3\alpha\left(\frac{a+\overline{v}}{3},\frac{a^2+3a+3+\overline{w}}{3}\right)\succ 1 \text{ and } 2\alpha\left(\frac{a+\overline{v}}{3},\frac{a^2+3a+3+\overline{w}}{3}\right)\not\succeq 1.
\]
\item Let $v_1, v_2\in\Z$ be such that $0\leq v_1\leq a$, $0\leq v_2\leq a-2$, $\frac{3a-1}{2}<v_1+v_2$ and $\frac{v_1+v_2-1}{3}\in\Z$. Then there exist $w_1,w_2\in\Z$, $v_1(a+2)+1\leq w_1\leq (a+1)(v_1+1)$, $v_2(a+2)+1\leq w_2\leq (a+1)(v_2+1)$ such that $\frac{w_1+w_2}{3}\in\Z$,
\[
3\alpha\left(\frac{v_1+v_2-1}{3},\frac{w_1+w_2}{3}\right)\succ 1 \text{ and } 2\alpha\left(\frac{v_1+v_2-1}{3},\frac{w_1+w_2}{3}\right)\not\succeq 1.
\]  
\end{enumerate}
Moreover for every $\overline{v}$ as in (1) and every $v_1$ and $v_2$ as in (2) such that $v_1+v_2-1\neq 2a-3$, we can find $\overline{w},w_1$ and $w_2$ as above such that
\[
3\alpha\left(\frac{a+\overline{v}}{3},\frac{a^2+3a+3+\overline{w}}{3}\right)\neq 3\alpha\left(\frac{v_1+v_2-1}{3},\frac{w_1+w_2}{3}\right) . 
\]
\end{prop}  

\begin{proof}
Let us start with the first part. We see in the interval $[\overline{v}(a+2),(a+1)(\overline{v}+1)]$ for $\overline{w}$, there are at least 
$a-\overline{v}+1$ integers, which can be bounded from below by $3$ for $\overline{v}\leq a-2$. Thus, for every $\overline{v}\leq a-2$, we can choose $\overline{w}$ such that $\frac{a^2+3a+3+\overline{w}}{3}\in\Z$. Then, it is not difficult to see that
\begin{multline*}
3\alpha\left(\frac{a+\overline{v}}{3},\frac{a^2+3a+3+\overline{w}}{3}\right)-1=\alpha(\overline{v},\overline{w})+(-a-1-(a^2+3a+3)\rho+(a+2)\rho^2)\\=\alpha(\overline{v},\overline{w})+(\rho')^{-2}\succ 0,
\end{multline*}
where the expression for $(\rho')^{-2}$ was showed, e.g., in \cite{KT}. That proves 
\[
3\alpha\left(\frac{a+\overline{v}}{3},\frac{a^2+3a+3+\overline{w}}{3}\right)\succ 1.
\]

To show the second part of the statement for these elements, we will consider the codifferent of $K$. By this, we mean the set
\[
\O_K^{\vee}=\{\delta\in K;\text{Tr}_{K/\Q}(\alpha\delta)\in\Z\text{ for all }\alpha\in\O_K\}.
\]
By \cite[Section 5.1]{KT}, we know that the element
\[
    \tilde{\delta}=\frac{1}{a^2+3a+9}\left(-a-4-(2a+1)\rho+2\rho^2\right)\left(-a-2-a\rho+\rho^2\right)
\]
lies in $\O_K^{\vee}$ for $\O_K=\Z[\rho]$, and, moreover, it is totally positive. Therefore, if $\alpha\in\O_K^{+}\cup\{0\}$, then $\text{Tr}_{K/\Q}(\alpha\tilde{\delta})\geq 0$. It is rather a computational task to derive that
\begin{equation} \label{eq:codi1}
\text{Tr}_{K/\Q}\left(\left(2\alpha\left(\frac{a+\overline{v}}{3},\frac{a^2+3a+3+\overline{w}}{3}\right)-1\right)\rho'^2\tilde{\delta}\right)=1-\frac{2a}{3}-\frac{a^2}{3}-\frac{\overline{v}}{3}(2a^2+4a+2)+\frac{2a\overline{w}}{3}.
\end{equation} 
Now, it is important to realize that $\frac{a+\overline{v}}{3}\in\Z$ and $\frac{a^2+3a+3+\overline{w}}{3}\in\Z$ cannot be both satisfied for $\overline{w}=(a+1)(\overline{v}+1)$. Thus, we can assume $\overline{w}\leq (a+1)(\overline{v}+1)-1$, and we use this border value to find an upper bound on the right side of (\ref{eq:codi1}), which is
\begin{equation} \label{eq:codi1a}
1-\frac{2a}{3}+\frac{a^2}{3}-\frac{\overline{v}}{3}(2a+2).
\end{equation}
Here, we apply a lower bound $\frac{a-3}{2}+\frac{1}{2}\leq\overline{v}$ to estimate (\ref{eq:codi1a}) from above by $\frac{5}{3}-\frac{a}{3}$. However, that is negative for $a\geq 6$, which implies
\[
2\alpha\left(\frac{a+\overline{v}}{3},\frac{a^2+3a+3+\overline{w}}{3}\right)\not\succeq 1.
\] 

Now, let us proceed with the second part of the statement. Suppose that $v_1$ and $v_2$ satisfy the conditions, and let us take $w_1$ such that $v_1(a+2)+1\leq w_1\leq (a+1)(v_1+1)$. Then, since $0\leq v_2\leq a-2$, the interval for $w_2$ contains at least $3$ numbers. Therefore, we can pick one such that $\frac{w_1+w_2}{3}\in\Z$. Then
\[
3\alpha\left(\frac{v_1+v_2-1}{3},\frac{w_1+w_2}{3}\right)-1=\alpha(v_1,w_1)+\alpha(v_2,w_2).
\]
Since $\alpha(v_1,w_1),\alpha(v_2,w_2)\in\bt$, we have $\alpha(v_1,w_1)+\alpha(v_2,w_2)\succ 0$, from which the statement follows.

Now, it suffices to show that
\begin{equation}
2\alpha\left(\frac{v_1+v_2-1}{3},\frac{w_1+w_2}{3}\right)\not\succeq 1.
\end{equation}
However, using $w_1+w_2\leq (a+1)(v_1+v_2+2)$ and $\frac{3a-1}{2}+\frac{1}{2}\leq v_1+v_2$, we can follow the same procedure as before to show that
\[
\text{Tr}_{K/\Q}\left(\left(2\alpha\left(\frac{v_1+v_2-1}{3},\frac{w_1+w_2}{3}\right)-1\right)\rho'^2\tilde{\delta}\right)<\frac{5}{3}-\frac{a}{3}<0
\]
for all $a\geq 6$. 

Finally, we will show that for almost all choices of $\overline{v},v_1$ and $v_2$, we can find $\overline{w}, w_1$ and $w_2$ so that the above two elements do not coincide. First, if $v_1+v_2-1<a$, then we cannot have $v_1+v_2-1=a+\overline{v}$. So, in these cases, we are done. Thus, suppose $v_1+v_2-1\geq a$. If $v_1+v_2-1\leq 2a-6$, then we have at least $6$ integers in the interval for $w_1+w_2$, which is $[(v_1+v_2)(a+2)+2,(v_1+v_2+2)(a+1)]$, originating from the individual conditions on $w_1$ and $w_2$. Therefore, we have at least $2$ possibilities in the corresponding congruence class modulo $3$, so we can choose one for which $w_1+w_2\neq a^2+3a+3+\overline{w}$.

It remains to discuss $ v_1+v_2-1\in\{2a-5,2a-4\}$ as $v_1+v_2-1=2a-3$ is excluded in the statement. If $v_1+v_2-1=2a-5=a+\overline{v}$, it implies $\overline{v}=a-5$. Then 
\[
2a^2 - 6 \leq w_1 + w_2 \leq 2a^2 - 2.
\]
That gives us at least two possibilities for $w_1+w_2$ in the same congruence class modulo $3$ unless $w_1+w_2=2a^2-4$. If $w_1+w_2=2a^2-4$, then $w_1+w_2=a^2+3a+3+\overline{w}$ if $\overline{w}=a^2-3a-7$. However, we see that possible choices of $\overline{w}$ lie in the interval
\[
a^2-3a-9\leq \overline{w}\leq a^2-3a-4.
\]
For $a^2-3a-7$, we see that $a^2-3a-4$ is in the same congruence class modulo $3$. Thus, for $w_1+w_2=2a^2-4$, we can put $\overline{w}=a^2-3a-4$ to ensure that our elements are not equal to each other.

If $v_1+v_2-1=2a-4=a+\overline{v}$, then $\overline{v}=a-4$. Similarly, as before, we have
\[
2a^2 + a - 4 \leq w_1 + w_2 \leq 2a^2 + a - 1.
\]
This time, only $2a^2 + a - 4$ and $2a^2 + a - 1$ lie in the same congruence class modulo $3$, and for them, we can pick the other one so that $w_1+w_2\neq a^2+3a+3+\overline{w}$. For $\overline{v}=a-4$, the coefficient $\overline{w}$ belongs to the interval
\[
a^2 - 2a - 7 \leq \overline{w} \leq a^2 - 2a - 3.
\]
If $w_1+w_2=2a^2+a-3$, we put $\overline{w}=a^2-2a-3$. The case $w_1+w_2=2a^2+a-2$ is not possible since it is never divisible by $3$. That finishes the proof.        
\end{proof}

Now, we will provide a lower bound on the number of additively indecomposable binary quadratic forms in the simplest cubic fields. For them, we use elements appearing in the previous propositions.

\begin{theorem}[= Theorem \ref{thm:simplest_lower_bound}]
Let $\Q(\rho)$ be a simplest cubic field with $\O_K=\Z[\rho]$ such that $a\geq 6$, and let $a=6A+a_0$ where $a_0\in\{0,1,2,3,4,5\}$. Up to equivalence, the number of additively indecomposable binary quadratic forms in $\Q(\rho)$ is at least:
\begin{enumerate}
\item $\frac{27A^2+21A-6}{2}$ if $a_0=0$,
\item $\frac{27A^2+39A}{2}$ if $a_0=1$,
\item $\frac{27A^2+39A}{2}$ if $a_0=2$,
\item $\frac{27A^2+57A+12}{2}$ if $a_0=3$,
\item $\frac{27A^2+57A+18}{2}$ if $a_0=4$,
\item $\frac{27A^2+75A+36}{2}$ if $a_0=5$.
\end{enumerate}
\end{theorem}

\begin{proof}
In this proof, we will count additively indecomposable forms $2x^2+2xy+\alpha(v,w)y^2$ and $3x^2+2xy+\alpha(v,w)y^2$, and we will also include those which are obtained from these after transformations $T_1$ and $T_2$. For that, we use the following tools. First, we know that if two forms $Q_1$ and $Q_2$ are equivalent, then their determinants are associated, i.e., $\det(Q_1)=\det(Q_2)\varepsilon$ for some totally positive unit $\varepsilon$. Thus, for these forms, we show the opposite. For that, we use the fact that the determinants of the above forms lie strictly inside a fundamental domain for the action of multiplication by totally positive units. It means that if two elements from it are associated, they are necessarily equal to each other. More specifically, we consider the fundamental domain
\[
\R_0^++\R_0^+\rho^2+\R_0^2(\rho')^{-2}=\R_0^++\R_0^+\rho^2+\R_0^+(-(a+1)-(a^2+3a+3)\rho+(a+2)\rho^2),
\]
see \cite[Subsection 4.1]{KT}. 
It is important to mention that elements from $\bt$ lie inside this domain.

First, let us consider forms $Q(x,y)=2x^2+2xy+\alpha(v,w)y^2$ where $0\leq v\leq \frac{a-1}{2}$ and $\frac{(2v+1)(a+2)+1}{2}\leq w\leq (v+1)(a+1)$, which are additively indecomposable by Propositions~\ref{prop:2alpha} and \ref{prop:form_of_form23}. We know that their determinant is
\begin{equation} \label{eq:det1}
\det(Q)=2\alpha(v,w)-1=-(2v+1)-2w\rho+(2v+2)\rho^2\in\bt.
\end{equation}
Thus, for different choices of $v$ and $w$, we get different elements of $\bt$, so these forms are not mutually equivalent. Moreover, let us also consider the forms $2(\rho')^{-2}x^2+2(\rho')^{-2}xy+(\rho')^{-2}\alpha'(v,w)y^2$ and $2\rho^2x^2+2\rho^2xy+\rho^2\alpha''(v,w)y^2$, which are also additively indecomposable. Easily, their determinants are 
\[
(\rho')^{-2}(\det(Q))'=2(\rho')^{-2}\alpha'(v,w)-(\rho')^{-2}
\] 
and
\[
\rho^2(\det(Q))''=2\rho^2\alpha''(v,w)-\rho^2.
\]
From the properties of transformations $T_1$ and $T_2$, we know that $(\rho')^{-2}\alpha'(v,w),\rho^2\alpha''(v,w)\in\bt$, and let us put $\alpha(v_1,w_1)=(\rho')^{-2}\alpha'(v,w)$ and $\alpha(v_2,w_2)=\rho^2\alpha''(v,w)$. It suffices to check that $2\alpha(v,w)-1$, $2\alpha(v_1,w_1)-(\rho')^{-2}$ and $2\alpha(v_2,w_2)-\rho^2$ cannot be equal to each other for different choices of $v,w,v_1,w_1,v_2$ and $w_2$. However, we have
\begin{equation} \label{eq:det2}
2\alpha(v_1,w_1)-(\rho')^{-2}=-(2v_1-a-1)-(2w_1-a^2-3a-3)\rho+(2v_1-a)\rho^2
\end{equation}
and
\begin{equation} \label{eq:det3}
2\alpha(v_2,w_2)-\rho^2=-2v_2-2w_2\rho+(2v_2+1)\rho^2.
\end{equation} 
We see that the coefficient before $\rho$ in (\ref{eq:det2}) is odd, and the same coefficient for (\ref{eq:det1}) and (\ref{eq:det3}) is always even, so an element in (\ref{eq:det2}) cannot be equal to elements in (\ref{eq:det1}) and (\ref{eq:det3}). Similarly, the coefficient before $\rho^2$ in (\ref{eq:det1}) is even while its value in (\ref{eq:det3}) is odd, which excludes equality between elements in (\ref{eq:det1}) and (\ref{eq:det3}). Therefore, forms from these three sets are mutually non-equivalent. 

Now, we will count their number. Recall that $0\leq v\leq \frac{a-1}{2}$ and $\frac{(2v+1)(a+2)+1}{2}\leq w\leq (v+1)(a+1)$. In overall (including forms after $T_1$ and $T_2$), we obtain $\frac{3}{8}a(a+2)$ forms for $a$ even, and $\frac{3}{8}(a+1)(a+3)$ forms for $a$ odd.

We will proceed with forms of the type $3x^2+2xy+\alpha(v,w)y^2$ where $\alpha(v,w)$ is as in Proposition \ref{prop:3alpha}. First, note that the determinant of these forms is a decomposable integer, so they cannot be equivalent to the forms from the previous part, whose determinant is an indecomposable integer. As before, we will also consider the forms appearing after the application of transformations $T_1$ and $T_2$. Note that all of $3x^2+2xy+\alpha(v,w)y^2$, $3(\rho')^{-2}x^2+2(\rho')^{-2}xy+(\rho')^{-2}\alpha'(v,w)y^2$ and $3\rho^2x^2+2\rho^2xy+\rho^2\alpha''(v,w)y^2$ have a determinant of the form $\alpha(\overline{v},\overline{w})+\varepsilon$ or $\alpha(v_1,w_1)+\alpha(v_2,w_2)$ where $\alpha(\overline{v},\overline{w}),\alpha(v_1,w_1),\alpha(v_2,w_2)\in\bt$ and $\varepsilon\in\{1,\rho^2,(\rho')^{-2}\}$, and since all of these elements lie inside $\R_0^++\R_0^+\rho^2+\R_0^2(\rho')^{-2}$, the determinants of our forms also belongs to this fundamental domain for the action of multiplication by totally positive units. Therefore, it is enough to show that these elements are not equal to each other. Recall that by Proposition \ref{prop:3alpha}, we can choose $\overline{w},w_1$ and $w_2$ so that $\alpha(v_1,w_1)+\alpha(v_2,w_2)\neq \alpha(\overline{v},\overline{w})+(\rho')^{-2}$ if $v_1+v_2-1=a+\overline{v}$. Note that this property is transmitted if we use $T_1$ and $T_2$.

Therefore, we need to study determinants
\begin{equation} \label{eq:3_1}
3\alpha(v,w)-1=-(3v+1)-3w\rho+(3v+3)\rho^2,
\end{equation}
\begin{equation} \label{eq:3_2}
3\alpha(v,w)-(\rho')^{-2}=-(3v-a-1)-(3w-a^2-3a-3)\rho+(3v-a+1)\rho^2
\end{equation}
and
\begin{equation} \label{eq:3_3}
3\alpha(v,w)-\rho^2=-3v-3w\rho^2+(3v+2)\rho^2.
\end{equation} 
In Table \ref{tab:coef_mod_3}, for elements in (\ref{eq:3_1}), (\ref{eq:3_2}) and (\ref{eq:3_3}), we show values of coefficients in basis $\{1,\rho,\rho^2\}$ modulo $3$ for different choices of $a$. We see that in all these cases, the values modulo $3$ are different for at least one coefficient, so elements in (\ref{eq:3_1}), (\ref{eq:3_2}), and (\ref{eq:3_3}) cannot be equal to each other. That implies that our forms are indeed non-equivalent.

\begin{table}
\centering
\begin{tabular}{c|c|c|c|c|c|c|c|c|c|c|}
\cline{2-10}
& \multicolumn{3}{|c|}{$3\alpha(v,w)-1$}& \multicolumn{3}{|c|}{$3\alpha(v,w)-(\rho')^{-2}$} & \multicolumn{3}{|c|}{$3\alpha(v,w)-\rho^2$}\\
\hline
\multicolumn{1}{|c|}{$a\;(\text{mod }3)$} & \hspace*{0.75mm} $1$ \hspace*{0.75mm} & \hspace*{0.75mm} $\rho$ \hspace*{0.75mm} & \hspace*{0.75mm} $\rho^2$ \hspace*{0.75mm} & \hspace*{0.75mm} $1$ \hspace*{0.75mm} & \hspace*{0.75mm} $\rho$ \hspace*{0.75mm} & \hspace*{0.75mm} $\rho^2$ \hspace*{0.75mm} & \hspace*{0.75mm} $1$ \hspace*{0.75mm} & \hspace*{0.75mm} $\rho$ \hspace*{0.75mm} & \hspace*{0.75mm} $\rho^2$ \hspace*{0.75mm}\\
\hline
\multicolumn{1}{|c|}{$0$} & $2$ & $0$ & $0$ & $1$ & $0$ & $1$ & $0$ & $0$ & $2$\\
\hline 
\multicolumn{1}{|c|}{$1$} & $2$ & $0$ & $0$ & $2$ & $1$ & $0$ & $0$ & $0$ & $2$\\
\hline
\multicolumn{1}{|c|}{$2$} & $2$ & $0$ & $0$ & $0$ & $1$ & $2$ & $0$ & $0$ & $2$\\
\hline
\end{tabular}
\caption{Values modulo $3$ of coefficients in base $\{1,\rho,\rho^2\}$ of elements $3\alpha(v,w)-1$, $3\alpha(v,w)-(\rho')^{-2}$ and $3\alpha(v,w)-\rho^2$} \label{tab:coef_mod_3}
\end{table}

It remains to count them. First, for elements from Proposition \ref{prop:3alpha}(1), we need count all $\overline{v}$ such that $\frac{a-3}{2}<\overline{v}\leq a-2$ and, moreover, we need $\frac{a+\overline{v}}{3}\in\Z$, i.e., we can include only one third of the above interval. Let us write $a$ as $a=6A+a_0$, where $A\in\N_0$ and $a_0\in\{0,1,2,3,4,5\}$. Then, it is easy to show that the number of additively indecomposable quadratic forms originating from these elements (three times to add forms after $T_1$ and $T_2$) is at least $3A$.  

Now, we turn our attention to elements in Proposition \ref{prop:3alpha}(2). Here, we will count different values of $v_1+v_2$ where $\frac{3a-1}{2}<v_1+v_2\leq 2a-3$. Note that every integer between these two bounds is attained for at least one pair $(v_1,v_2)$. The case $v_1+v_2=2a-2$ is excluded, and moreover, we must again consider only one third of these numbers. Therefore, we get the following lower bounds:
\begin{enumerate}
\item $3(A-1)$ for $a_0=0,1,2,3$,
\item $3A$ for $a_0=4,5$. 
\end{enumerate}
Summing everything up, we get the values in the statement of the theorem.     
\end{proof}  

\section{Declaration}
\textbf{Acknowledgment.} We thank the number theory group at Charles University for useful discussions on this topic.  

\textbf{Conflict of Interest/Competing Interest.} None. 

\textbf{Funding statement.} The authors are supported by Czech Science Foundation GAČR, grant 22-11563O. Simona Fryšová was further supported by Charles University programmes PRIMUS/24SCI/010, GA UK 252931 and SVV-2025-260837.

\medskip


\begin{thebibliography}{10}  

\bibitem{BI} R. Baeza and M. I. Icaza, \textit{Decomposition of positive definite integral quadratic forms as sums of positive definite quadratic forms}, Proc. Sympos. Pure Math. 58, 63--72 (1995).

\bibitem{CLSTZ} M. \v{C}ech, D. Lachman, J. Svoboda, M. Tinkov\'a and K. Zemkov\'a, \textit{Universal quadratic forms and indecomposables over biquadratic fields}, Math. Nachr. 292, 540--555 (2019).

\bibitem{EK1} P. Erd\"os and C. Ko, \textit{Some Results on Definite Quadratic Forms}, J. London Math. Soc. 13(3), 217--224 (1938). 

\bibitem{EK2}  P. Erd\"os and C. Ko, \textit{On definite quadratic forms, which are not the sum of two definite or semi-definite forms}, Acta Arith. 3, 102--122 (1939). 


\bibitem{Hla} S. Hlavinkov\'a, \textit{Non-decomposable binary quadratic forms over number fields}, Master’s thesis, Charles University, 2025.

\bibitem{Ja} F. Jarvis, \textit{Algebraic Number Theory}, Springer, 2007.

\bibitem{KT} V. Kala and M. Tinkov\'a, \textit{Universal quadratic forms, small norms and traces in families of number fields}, Int. Math. Res. Not. IMRN 2023, 7541-7577 (2023).

\bibitem{KY1} V. Kala and P. Yatsyna, \textit{Lifting problem for universal quadratic forms}, Adv. Math. 377, 107497 (2021).

\bibitem{KTZ} J. Krásenský, M. Tinková and K. Zemková, \textit{There are no universal ternary quadratic forms over biquadratic fields}, Proc. Edinb. Math. Soc. 63 (3), 861-912 (2020).

\bibitem{LP} F. Lemmermeyer and A. Peth\"{o}, \textit{Simplest Cubic Fields},  Manuscripta Math. 88, 53--58 (1995).

\bibitem{Man} S. H. Man, \textit{Minimal rank of universal lattices and number of indecomposable elements in real multiquadratic fields}, Adv. Math. 447, 109694 (2024). 

\bibitem{Mo2} L. J. Mordell, \textit{The representation of a definite quadratic form as a sum of two others}, Ann. of Math. 38(4), 751--757 (1937).  

\bibitem{Sh} D. Shanks, \textit{The simplest cubic number fields}, Math. Comp. 28, 1137--1152 (1974).


\bibitem{TY} M. Tinkov\'a and P. Yatsyna, \textit{Additively indecomposable quadratic forms over totally real number fields},  J. Lond. Math. Soc. 112(5), e70350 (2025).

\bibitem{Wi} K. S. Williams, \textit{Integers of biquadratic fields}, Canad. Math. Bull. 13, 519--526 (1970).

\end{thebibliography}
\end{document}